\documentclass[12pt,oneside]{amsart} 
\usepackage{amssymb,amscd}
\usepackage{euscript}
\usepackage{hyperref}

      \theoremstyle{plain}
      \newtheorem{theorem}{Theorem}[section]
      \newtheorem{lemma}[theorem]{Lemma}
      \newtheorem{corollary}[theorem]{Corollary}
      \newtheorem{proposition}[theorem]{Proposition}
      \newtheorem{remark}[theorem]{Remark}
      
      \newtheorem{definition}[theorem]{Definition}        
          
\numberwithin{equation}{section}

      \makeatletter
      \def\@setcopyright{}
      \def\serieslogo@{}
      \makeatother

\def\E{\mathcal{E}}
\def\U{\mathcal{U}}
\def\F{\tilde F}
\def\G{\mathcal{G}}
\def\M{\mathcal{M}}
\def\I{{I}}
\def\m{\EuScript{M}}
\def\W{\mathcal{W}}
\def\O{\mathcal O}
\def\c{\EuScript{C}}
\def\R{\mathbb R}
\def\rd{{\mathbb R ^d}}   

\def\Z{\mathbb Z}
\def\N{\mathbb N}
\def\T{\mathbb T}

\def\dist{\text{dist}}

\def\Id{\text{Id}}
\def\e{\epsilon}
\def\bv{\mathbf v}
\def\a{\alpha}

\def\b{\beta}

\def\QED{\hfill\hfill{\square}}

\begin{document}

\date{\today}
\author{Boris Kalinin$^\ast$ and Victoria Sadovskaya$^{\ast\ast}$}

\address{Department of Mathematics, The Pennsylvania State Ubiversity, University Park, PA 16802, USA.}
\email{kalinin@psu.edu, sadovskaya@psu.edu}

\title [Cocycles with one exponent over partially hyperbolic systems]
{Cocycles with one exponent over partially hyperbolic systems} 

\thanks{$^{\ast}$  Supported in part by NSF grant DMS-1101150}
\thanks{$^{\ast\ast}$ Supported in part by NSF grant DMS-0901842}

%%%%%%%%%%%%%%%%%%%%%%%%%%%%%%%

\begin{abstract}
We consider H\"older continuous linear cocycles over partially 
hyperbolic diffeomorphisms. For fiber bunched cocycles with 
one Lyapunov exponent we show continuity of measurable 
invariant conformal structures and sub-bundles. Further, we
establish a continuous version of Zimmer's Amenable Reduction
Theorem. For cocycles over hyperbolic systems we also obtain 
polynomial growth estimates for the norm and quasiconformal 
distortion from the periodic data.
\end{abstract}

\maketitle 

%%%%%%%%%%%%%%%%%%%%%%%%%%%%
%%%%%%%%%% Introduction %%%%%%%%%%%
%%%%%%%%%%%%%%%%%%%%%%%%%%%%
 
 \section{Introduction}
A linear cocycle over a dynamical system $f: \M \to \M$ is an 
automorphism $F$ of a vector bundle $\E$ over $\M$ that covers $f$. 
In the case  of a trivial vector bundle $\M \times \rd$, a linear cocycle 
can be represented by a matrix-valued function
$A: \M \to GL(d,\R)$ via $F(x,v) = (f(x), A(x)v)$. 
In smooth dynamics linear cocycles arise naturally from 
the derivative. They play an important role in the study of smooth systems and 
group actions, especially in aspects related to rigidity. 

In this paper we consider H\"older continuous linear cocycles with 
one Lyapunov exponent over 
hyperbolic and partially hyperbolic diffeomorphisms. An important 
motivation comes from the restriction of the derivative to H\"older 
continuous invariant sub-bundles such as center, stable, and 
unstable bundles. In hyperbolic case we studied such cocycles
in \cite{KS9,KS10}. We concentrated on obtaining conformality 
of the cocycle from its periodic data and applying this to local 
and global rigidity of Anosov systems \cite{KS9,GKS11}. 
From a different angle, such cocycles over hyperbolic and 
partially hyperbolic systems were considered in \cite{V,ASV}.
In particular, it was shown that cocycles with more than 
one Lyapunov exponent are generic in various cases, for example
in a neighborhood of a fiber bunched cocycle. These results 
indicated that having one exponent is an exceptional property. 
In this paper we show that it is true in a very strong sense
by developing a structural theory for such cocycles. We expect
that these results will be useful in the study of partially hyperbolic
systems and in the area of rigidity of hyperbolic systems and actions.

In the base we consider a partially hyperbolic diffeomorphism $f$ 
which is is volume-preserving, accessible, and center bunched.
This is the same setting as in the latest results on ergodicity of
partially hyperbolic diffeomorphisms \cite{BW}, except that we 
require accessibility instead of essential accessibility. We assume that the 
cocycle $F$ over $f$ is fiber bunched, i.e. non-conformality of $F$ 
in the fiber is dominated by the expansion/contraction along the 
stable/unstable foliations of $f$ in the base. This or similar conditions 
play a role in all results on noncommutative cocycles over
hyperbolic or partially hyperbolic systems. If $F$ is the restriction of 
the derivative of $f$ to the center sub-bundle, fiber bunching for $F$ 
corresponds to the strong center bunching for $f$. Thus our results 
apply to this setup.

For fiber bunched cocycles with one Lyapunov exponent with respect
to the volume, we prove continuity of measurable $F$-invariant  
sub-bundles and conformal structures. We use this to establish a 
continuous version 
of Zimmer's Amenable Reduction Theorem. Passing to a finite cover 
and a power of $F$, if necessary, we show  existence of a continuous 
flag of sub-bundles such that the induced cocycles on the factor bundles 
are conformal. For cocycles over hyperbolic systems we obtain stronger results including
H\"older regularity of the invariant structures.
In particular, for cocycles with one exponent at  each periodic orbit we
obtain the reduction for $F$ itself and polynomial growth estimates 
for its quasiconformal distortion. 

We formulate the results in Section \ref{statements} and 
give the proofs in Section \ref{proofs}.

%%%%%%%%%%%%%%%%%%%%%%%%%%

%%%%%%%%%%%%%%%%%%%%%%%%%%%%%
%%%%%%  Definitions and notations  %%%%%
%%%%%%%%%%%%%%%%%%%%%%%%%%%%%

\section{Definitions and notations} \label{preliminaries}

In this parer  $\M$ denotes a  compact connected smooth manifold.
 
 \subsection{Partially hyperbolic diffeomorphisms.} \label{partial}
 (See  \cite{BW} for more details.)
 
 A diffeomorphism $f$ of $ \M$ 
is said to be {\em partially hyperbolic} if
there exist a nontrivial $Df$-invariant splitting of the tangent bundle 
$T\M =E^s\oplus E^c \oplus E^u,$ and  a Riemannian 
metric on $\M$ for which one can choose continuous positive 
functions $\nu<1,\,$ $\hat\nu<1,\,$ $\gamma,$ $\hat\gamma\,$ such that 
for any $x \in \M$ and unit vectors  
$\,\bv^s\in E^s(x)$, $\,\bv^c\in E^c(x)$, and $\,\bv^u\in E^u(x)$
\begin{equation}\label{partial def}
\|Df(\bv^s)\| < \nu(x) <\gamma(x) <\|Df(\bv^c)\| < \hat\gamma(x)^{-1} <
\hat\nu(x)^{-1} <\|Df(\bv^u)\|.
\end{equation}
The sub-bundles $E^s$, $E^u$, and $E^c$ are called, respectively,
 stable, unstable, and center.
$E^s$ and $E^u$  are 
tangent to the stable and unstable foliations $W^s$ and $W^u$
 respectively.  An $su$-path in $\M$ is a concatenation 
 of finitely many subpaths which lie entirely in a single 
 leaf of $W^s$ or  $W^u$. A partially hyperbolic  
 diffeomorphism $f$  is called {\em accessible}  if any two points 
 in $\M$ can be connected by an $su$-path.

We say that $f$ is {\em volume-preserving} if it has an invariant probability 
measure $\mu$ in the measure class of a volume induced by a 
Riemannian metric (the density of $\mu$ is not required to be smooth). It is conjectured that any essentially accessible 
$f$ is ergodic with respect to such  $\mu$. This was established 
if $f$ is $C^2$ and center bunched \cite{BW}.
The diffeomorphism $f$ is called {\em center bunched}\,
if the functions  $\nu, \hat\nu, \gamma, \hat\gamma$  can be 
chosen to satisfy
\begin{equation}\label{center bunching}
\nu<\gamma \hat \gamma \quad \text{and}
\quad \hat\nu<\gamma \hat \gamma .
\end{equation}
This implies that 
$\|Df|_{E^c}\|\cdot\|(Df|_{E^c})^{-1}\|$,
which is a measure of  non-con\-for\-mality of $f$ on $E^c$, 
is dominated by the
contraction on $E^s$ and expansion on $E^u$.

If $f$ is $C^{1+\delta}$, the ergodicity holds under
{\em strong center bunching} assumption \cite{BW}:
 \begin{equation}\label{strong center bunching}
\nu^\theta<\gamma \hat \gamma \quad \text{and}
\quad \hat\nu^\theta<\gamma \hat \gamma,
\end{equation}
 where $\theta \in (0, \delta)$ satisfies
 \begin{equation}\label{theta}
\nu\gamma^{-1}< \kappa^\theta \quad \text{and}\quad
\hat\nu \hat\gamma^{-1}< \hat \kappa^\theta,
\end{equation}
 for some functions $\kappa$ and $\hat \kappa$ such that for all $x$ in $\M$
$$
\kappa(x)<\|Df(v)\|  \;\text{ if }\, v \in E^s(x)
\quad\text{and}\quad
\|Df(v)\| < \hat\kappa (x)^{-1}\text{ if }\, v \in E^u(x).
$$
 It is known  that the first inequality in \eqref{theta}
implies that $E^c\oplus E^s$ is $\theta$-H\"older, 
 the second one yields the same for $E^c\oplus E^s$,
and thus \eqref{theta} implies that $E^c$  is $\theta$-H\"older.

 \subsection{H\"older continuous vector bundles and linear cocycles}\label{bundle} $\;$
 
\noindent We consider  a finite dimensional  $\beta$-H\"older, $0<\beta \le 1$,
  vector bundle  $P : \E \to \M $. This means that there exists
an open cover $\{ U_i \}_{i=1}^k$ of $\M$ with coordinate systems 
$$
\phi_i : P^{-1} (U_i) \to U_i \times \rd, \quad \phi_i (v) = (P(v),\Phi_i(v))
$$
such that $\phi_j \circ \phi_i^{-1}$ is a homeomorphism and its
 restriction to the fiber $L_x=\Phi_j \circ \Phi_i^{-1}|_{\{x\} \times \R^d}$ depends $\beta$-H\"older on $x$, i.e. there is $C$  such that
  $\| L_x - L_y \| \le C \, \dist (x,y) ^\beta 
 $ for all $i,j$ and all $x,y \in U_i \cap U_j$. 
One can realize $\E$  as a $\beta$-H\"older sub-bundle of a trivial bundle
 by $\phi : \E \to \M  \times \R^{kd}$ with 
 $\phi (v) = (P(v), \rho_1 \Phi_1(v) \times ... \times \rho_k \Phi_k(v))$,
where $\{ \rho_i \}$ is a $\beta$-H\"older partition of unity for $\{ U_i \}$.

Using such an embedding we equip $\E$ with the induced $\beta$-H\"older  
Riemannian metric, i.e. a family of inner products on the fibers, and
fix an identification $\I_{xy} :\E_x \to \E_y$ of fibers at 
nearby points. We define the latter as 
$\Pi_y^{-1} \circ \Pi_x$, where $\Pi_x$ is the orthogonal projection in 
$\R^{kd}$ from $\E_x$ to the subspace which is the middle point of
the unique shortest geodesic between $\E_x$ and $\E_y$ in the 
Grassmannian of $d$-dimensional subspaces. 
The identifications $\{ \I_{xy} \}$ vary $\beta$-H\"older
on a neighborhood of the diagonal in $\M \times \M$ and  satisfy for
some constant $C$ and any unit vector $u \in \E_x$
\begin{equation}\label{I}
 \I_{xy}= \I_{yx}^{-1}, \;\; \| \I_{xy} u - u\| \le C \dist (x,y)^\beta ,
\; \text{and hence } | \| \I_{xy}\| -1| \le C \dist (x,y)^\beta.
\end{equation}
% \| \I_{xz}-\I_{yz} \circ \I_{xy}\|\le C \dist (x,y)^\beta

Let $f$ be a diffeomorphism of  $\M$
and $P : \E \to \M $ be a finite dimensional  $\beta$-H\"older 
vector bundle over $\M$.
A continuous linear cocycle over $f$ is a homeomorphism
$F:\E\to\E$ such that $P \circ F = f \circ P$ and $F_x : \E_x \to \E_{fx}$
is a linear isomorphism. Such an $F$ is called $\beta'$-H\"older, 
$0<\beta' \le \beta$, if $F_x$ depends $\beta'$-H\"older on $x$,
more precisely, if there exist $C$ such that 
for all nearby points $x,y\in \M$
\begin{equation}\label{FHolder}
 \| F_x - \I_{fxfy}^{-1} \circ F_y \circ \I_{xy}\|  \le C \dist (x,y)^{\beta'}.
\end{equation}

%%%%%%%%%%%%%%%%%%%%%%%%%%

\subsection{Conformal structures}\label{conf structure}
(See \cite{KS10} for more details.)
A conformal structure on $\R^d$, $d\geq 2$, is a class of proportional 
inner products. The space $\c^d$ of conformal structures on $\R^d$
can be identified with the space of real symmetric positive definite $d\times d$ 
matrices with determinant 1, which is isomorphic to $SL(d,\R) /SO(d,\R)$. 
The group $GL(d,\R)$ acts transitively on $\c^d$ via $X[C] = (\det X^TX)^{-1/d}\; X^T C \, X, $ and $\c^d$ carries a $GL(d,\R)$-invariant Riemannian metric of non-positive curvature.  The distance 
to the identity in this metric is
\begin{equation}\label{dist(id,C)}
\dist (\Id , C) = \sqrt{d}/2 \cdot \left( (\log \lambda_1)^2 + 
\dots +(\log \lambda_d)^2 \right) ^{1/2},
\end{equation}
where $\lambda_1, \dots, \lambda_d$ are the eigenvalues of $C$
(see \cite[p.327]{T} for more details and \cite[p.27]{M} for the formula).
%The distance between two structures $C_1$ and $C_2$ can be computed as $\dist (C_1, C_2)= \dist (\Id, X[C_2])$,  where $ X[C_1]=\Id.$ 

%It is easy to check the following relation between this metric and the operator norm \begin{equation}\label{compare}\sqrt{d/8}  \cdot  \log (\|C\|\cdot \|C^{-1}\|) \le \dist (\Id, C)\le d/2 \cdot \max\{ \log \|C\|,  \log \|C^{-1}\|\}.\end{equation}
%We also note that $\|C^{-1}\| \le \|C\|^{d-1}$. Thus a subset of $\c^d$ is bounded with respect to this distance if and only if it is bounded with respect to the operator norm. We also note that on any bounded subset of $\c^d$ this distance is bi-Lipschitz equivalent to the distance induced by the operator norm on matrices.

%\vskip.1cm

For a vector bundle  $\E \to \M$ we can consider a bundle $\c$ over $\M$ 
whose fiber $\c _x$ is the space of conformal structures on $\E_x$.
Using a background Riemannian metric on $\E$, the space $\c _x$ can 
be identified with the space of symmetric positive linear operators on $\E_x$ 
with determinant 1. We equip the fibers of $\c$ with the Riemannian 
metric as above.  A continuous (measurable) section of $\c$ is called 
a continuous  (measurable) conformal structure on $\E$. 

An invertible linear map $A :\E_x \to \E_y$ induces an isometry from 
$\c _x$ to $\c _y$ via $A(C) =(\det (A^*A))^{1/d}  (A^{-1})^*C (A^{-1})$, 
where $C$ is a conformal structure viewed as an operator.
If $F : \E \to \E$ is a linear cocycle over $f$, we say that a conformal 
structure $\tau$ on $\E$ is $F${\em -invariant}\, if $F(\tau(x)) = \tau(f(x))$
for all $x \in \M$.

%%%%%%%%%%%%%%%%%%%%%

 \section{Statements of  results} \label{statements}
 
\noindent{\bf Standing assumptions.} 
{\em Unless stated otherwise, in  this paper\\
$\M$ is a compact connected smooth manifold; \\
$f:\M\to\M$ is an accessible partially hyperbolic diffeomorphism 
that preserves a volume $\mu$ and is  
either $C^2$ and center bunched, or $C^{1+\delta}$ and strongly center bunched;\\
$\;P : \E \to \M $ is a finite dimensional  $\beta$-H\"older 
vector bundle over $\M$;  \\$F:\E\to\E$ is a $\beta$-H\"older
 linear cocycle over $f$.}
\vskip.4cm

First we establish continuity of measurable invariant conformal 
structures for fiber bunched cocycles. A cocycle $F$ over a partially 
hyperbolic diffeomorphism $f$ is called {\em fiber bunched}\, if for 
some $\beta$-H\"older norm on 
$\E$ 
\begin{equation}\label{fiber bunched}
\|F(x)\|\cdot \|F(x)^{-1}\| \cdot  \nu(x)^\beta <1 \quad\text{and}\quad
\|F(x)\|\cdot \|F(x)^{-1}\| \cdot  \hat\nu(x)^\beta <1
\end{equation}
for all $x$ in $\M$.
This condition  allows to establish convergence of
certain iterates of the cocycle along the  stable and unstable leaves. 

\begin{theorem} \label{structure} 
If $F$ is fiber bunched, then any $F$-invariant $\mu$-measurable 
conformal structure on $\E$  coincides $\mu$-a.e. with a continuous 
conformal structure. 

\end{theorem} 

We denote the iterate 
$F_{f^{n-1}x} \circ ... \circ F_{fx} \circ F_x\,$ by $\,F^n_x$.
A cocycle $F$ is called  {\em uniformly quasiconformal}\,
 if  the quasiconformal distortion 
\begin{equation}\label{K_F}
 K_F(x,n)\overset{\text{def}}{=} \|F^n_x\| \cdot \| (F^n_x)^{-1}  \|
 \end{equation}
  is uniformly bounded 
 for all $x\in \M$ and $n\in\Z$. The cocycle is said to be
  {\em conformal} with respect to some
Riemannian metric on $\E$ if $K_F(x, n)=1$ for all $x$ and $n$.
 
 \begin{corollary} \label{qc implies conf}
If $F$ is uniformly quasiconformal then it preserves a 
continuous conformal structure on $\E$, equivalently,
$F$ is conformal with respect to a continuous 
Riemannian metric on $\E$.

\end{corollary}

\vskip.1cm

Next we address continuity of measurable invariant sub-bundles.
If a cocycle has more than one Lyapunov exponent,
then the corresponding Lyapunov sub-bundles are 
invariant and measurable, but not continuous in general.
We show that for a fiber bunched cocycle with only one Lyapunov
exponent measurable invariant sub-bundles are continuous.
We denote by  $\lambda_+(F,\mu)$
and $\lambda_-(F,\mu)$ the largest and smallest Lyapunov 
exponents of $F$ with respect to $\mu$. We recall that
for  $\mu$  almost every $x\in \M$,  
\begin{equation} \label{exponents}
\lambda_+(F,\mu)= \lim_{n \to \infty} \frac 1n \log \| F_x ^n \| 
\quad \text{and}\quad
\lambda_-(F,\mu)=  \lim_{n \to \infty} \frac 1n \log \| (F_x ^n)^{-1} \|^{-1} 
\end{equation}
 (see \cite[Section 2.3]{BP} for more details).

\begin{theorem} \label{distribution} 
Suppose that $F$ is fiber bunched and $\lambda_+(F,\mu)=\lambda_-(F,\mu)$.
Then any $\mu$-measurable $F$-invariant sub-bundle of $\E$  coincides $\mu$-a.e. with a continuous one.
\end{theorem}

Using Theorems \ref{structure} and \ref{distribution} together
with Zimmer's Amenable Reduction Theorem we obtain the
following description of fiber bunched cocycles with one exponent.

For any finite cover $p: \tilde \M \to \M$ the pullback of $\E$ defines  
a $\beta$-H\"older vector bundle $\tilde \E$ over $\tilde \M$.
If $\tilde f : \tilde \M \to \tilde \M$ is a diffeomorphism covering $f$,
then $F$ lifts uniquely to  a $\beta$-H\"older linear cocycle 
$\tilde F : \tilde  \E \to \tilde  \E$ over $\tilde f$ that covers $F$.
We call such a cocycle $\tilde F$ a finite cover of $F$.
%i.e. $\bar p \circ \tilde F = F \circ \bar p$.

\begin{theorem}[Continuous  Amenable Reduction] \label{reduction} 
Suppose that $F$ is fiber bunched and
 $\lambda_+(F,\mu)=\lambda_-(F,\mu)$. Then there exists a finite
 cover $\tilde F : \tilde  \E \to \tilde  \E$ of $F$
and $N\in \N$ such that $\tilde F^N$ satisfies the following property.
There exist a flag of continuous $\tilde F^N$-invariant sub-bundles 
\begin{equation} \label{flag}
\{ 0 \}=\tilde \E^0 \subset \tilde \E^1 \subset ... \subset  \tilde \E^{k-1} \subset \tilde \E^k =\tilde \E
\end{equation} 
and continuous conformal structures
on the  factor bundles $ \,\tilde \E^{i}/ \tilde \E^{i-1}$,
$\,i=1, ... , k,$ invariant under  the factor-cocycles 
induced by $\tilde F^N$.
 \end{theorem} 

The proof shows that when it is necessary to pass to a cover, the resulting cocycle $\tilde F^N$ preserves more than one flag as in \eqref{flag}. Their union is preserved by $\tilde F$ and is the lift of an invariant object for $F$. 
To illustrate this, in Section \ref{example}  we construct
a cocycle $F$ on $\E=\T^2 \times \R^2$ with no invariant $\mu$-measurable sub-bundles or conformal structures. Its lift $\tilde F$ to a double cover preserves two continuous line bundles, while $F$ preserves a continuous 
field of pairs of lines.

In the case when $ \tilde \E_1=\tilde \E$, the cocycle $F$ itself is conformal 
on $\E$ with respect to some continuous Riemannian metric. This can
be easily seen from the proof or deduced from the theorem using 
Corollary~\ref{qc implies conf}.

If there are $d=\dim \E_x$ continuous vector fields which give bases 
for all $ \tilde \E^i$, then the theorem implies that $\tilde F^N$ is continuously
cohomologous to a cocycle with values in a ``standard" maximal amenable
subgroup of $GL(d,\R)$, see Remark \ref{triv}. However, triviality 
of $\tilde \E$ alone is insufficient for such reduction even if  $\tilde \E=\T^2 \times \R^2$ since invariant sub-bundles may be non-orientable, see 
\cite[Section 8.1]{S11} or the example in Section \ref{example}.

%and corresponds to  a smaller amenable group.

 \begin{remark} If $f$ is $C^{1+\delta}$ and satisfies the strong 
 center bunching condition, then the above results apply 
to $F=Df|_{E^c}$.
Indeed, the condition \eqref{theta} implies that $E^c$  is $\theta$-H\"older
and \eqref{strong center bunching} yields 
fiber bunching  \eqref{fiber bunched} for  $F=Df|_{E^c}$.
 \end{remark}

We also show that  fiber bunching  can be replaced by the following 
assumption (existence of the invariant volume $\mu$ is still assumed).

\begin{corollary} \label{reductioncor}
Suppose that $\lambda_+(F,\eta)=\lambda_-(F,\eta)$
for every ergodic $f$-invariant measure $\eta$. Then the conclusions
of Theorems \ref{structure},  \ref{distribution}, and  \ref{reduction} hold.  
\end{corollary}

This corollary relies on a certain estimate for subadditive 
sequences of continuous functions. In Proposition \ref{<0} we 
give a definitive version of this useful result.

\vskip.1cm

We obtain H\"older continuity of the invariant structures under 
a stronger accessibility
assumption. The diffeomorphism $f$ is said to be 
{\em locally $\alpha$-H\"older accessible}\,
if there exists a number $L=L(f)$ such that for all sufficiently close $x,y\in \M$ there is an $su$-path
$\{x=x_0, \;x_1, \;\dots ,\; x_L=y\}\,$ such that 
$$
  \dist_{W^i}  (x_{i-1}, x_i)\le C\,\dist (x,y)^\alpha \quad\text{for }i=1, \dots, L.
$$
Here the distance between $x_{i-1}$ and $x_i$ is measured  along the 
corresponding stable or unstable leaf $W^i$.

\begin{corollary} \label{locally accessible}
If $\,f$ is locally $\,\alpha$-H\"older accessible then the invariant conformal structures
and sub-bundles in Theorems \ref{structure}, \ref{distribution}, \ref{reduction}
and Corollaries \ref{qc implies conf}, \ref{reductioncor} are 
$\alpha \beta$-H\"older.
%The H\"older exponent is $\alpha \beta$, except for the invariant conformal structures on the sub-bundle and factor bundles in Theorem \ref{reduction}, for which it is $\alpha^2 \beta$.

\end{corollary}

Now we consider a special case when $f$ is a transitive $C^{1+\delta}$, 
$\delta>0$, Anosov diffeomorphism. This means that there is no center 
sub-bundle $E^c$ and thus the center bunching assumption is not needed. Due to the local product structure of the
stable and unstable manifolds, $f$ is locally $\alpha$-H\"older accessible
with $\a =1$, and hence the corollary above applies.  Moreover, we can 
take $\mu$ to be any ergodic measure with full support and local product structure. The latter means that $\mu$ is locally equivalent to the product 
of its conditional measures on the local stable and unstable manifolds. 
An invariant volume, if it exists, has these properties. Other examples 
include the measure of maximal entropy and equilibrium measures of 
H\"older continuous potentials. For the Anosov case, analogs of 
Theorem \ref{structure}, Corollary \ref{qc implies conf}, and of a weaker 
version of Theorem \ref{distribution} were obtained in \cite{KS10}. 

\begin{corollary} \label{anosov} 
Let $f$ be a transitive $C^{1+\delta}$ Anosov diffeomorphism, 
$\mu$ be an $f$-invariant ergodic measure with full support and local product structure, and $\M, \E, F$ be as in the standing assumptions.
Then Theorems \ref{structure}, \ref{distribution}, \ref{reduction}
and Corollary \ref{qc implies conf} hold and the resulting invariant 
conformal structures and sub-bundles are $\beta$-H\"older.

Moreover, the fiber bunching assumption in 
Theorems \ref{structure}, \ref{distribution}, \ref{reduction} can be replaced
by the assumption that for every $f$-periodic point $p$ the invariant measure 
$\mu_p$ on its orbit satisfies $\lambda_+(F,\mu_p)=\lambda_-(F,\mu_p)$.  

\end{corollary}

The assumption that there is only one exponent for every periodic 
measure implies the same for every invariant measure 
\cite[Theorem 1.4]{K11}. 
In this case we obtain further results. In the next theorem we construct an invariant
flag for $F$ itself and we normalize the cocycle and metrics so that
factor cocycles are isometries.

\begin{theorem} \label{reductionH}
 Let $f$ be a transitive  $C^{1+\delta}$ Anosov diffeomorphism.
Suppose that for every $f$-periodic point $p$ the invariant measure 
$\mu_p$ on its orbit satisfies $\lambda_+(F,\mu_p)=\lambda_-(F,\mu_p)$.
 Then there  exist a flag of $\beta$-H\"older  $F$-invariant sub-bundles 
\begin{equation} \label{flagH}
\{0\} =\E^0 \subset \E^1 \subset ...  \subset \E^{j-1} \subset \E^j = \E
\end{equation}
and $\beta$-H\"older Riemannian metrics on the factor 
bundles $\E^{i}/\E^{i-1}$, 
$i=1, ... , j$, so that for some positive $\beta$-H\"older function 
$\phi : \M \to \R$ the factor-cocycles induced by the cocycle 
$\phi F$ on $\E^{i}/\E^{i-1}$ are isometries.

 \end{theorem} 

The stronger conclusion relies on the property that the functions by 
which the cocycle rescales the conformal metrics on the factor bundles 
are cohomologus to each other. The condition that 
$\lambda_+(F,\eta)=\lambda_-(F,\eta)$ for every $f$-invariant measure $\eta$ 
is necessary for cohomology of these functions and for the conclusion 
of the theorem. 
This condition is not known to imply the cohomology for partially 
hyperbolic $f$, and thus we do not have an analog  of 
Theorem \ref{reductionH} for such $f$.

We use Theorem \ref{reductionH}  to obtain uniform polynomial growth 
estimates of the quasiconformal distortion $K_F(x,n)$ and the norm of $F$. 

\begin{theorem}[Polynomial Growth] \label{polynomial} 
Let $f$ be a transitive $C^{1+\delta}$ Anosov diffeomorphism.
Suppose that for every $f$-periodic point $p$ the invariant measure 
$\mu_p$ on its orbit satisfies $\lambda_+(F,\mu_p)=\lambda_-(F,\mu_p)$.
Then there exists 
$m<\dim \E _x$ and $C$ such that 
$$
K_F(x,n) \le Cn^{2m}\quad  \text{for all $x\in \M$ and $n\in \Z$. }
$$
 Moreover, if $\lambda_+(F,\mu_p)=\lambda_-(F,\mu_p)=0$ for every 
$\mu_p$, then there exists 
$m<\dim \E _x$ and $C$ such that 
$$
  \| F^n_x\| \le C|n|^{m}\quad  \text{for all $x\in \M$ and $n\in \Z$. }
  $$
 \end{theorem} 

\noindent One can take $m=j-1$, which is the number of
 non-trivial sub-bundles in  \eqref{flagH}.

%\newpage
%%%%%%%%%%%%%%%%%%%%%%%%%%%%%%%
%%%%%%%%%%%% Proofs %%%%%%%%%%%%%%%
%%%%%%%%%%%%%%%%%%%%%%%%%%%%%%%
%\newpage

\section{Proofs}\label{proofs}
 
%%%%%%%%%%%%%%%%%%%%%%%%%%%
%%% Technical result  %%%%%%%%
%%%%%%%%%%%%%%%%%%%%%%%%%%%

\subsection{Stable holonomies}
Convergence of  products of the type $(F^n_y)^{-1} \circ F^n_x$
has been observed for various types group-valued cocycles whose 
growth is slower than the expansion/contraction in the base 
(see e.g. \cite{NT,PW}).
It is also related to existence of strong stable/unstable
manifolds for the extended system on the bundle. 
We follow the notations and terminology form \cite{V,ASV} for 
linear cocycles, where it is more convenient to use the following
notion of holonomy.

\begin{definition} \label{holonomies}
A {\em stable holonomy} for a linear cocycle $F:\E \to \E$ is 
a continuous map $H^s:(x,y)\mapsto H^s_{xy}$, 
where $x\in \M$, $y\in W^s(x)$, such that  
\begin{itemize}
\item[(i)] $H^s_{xy}$ is a linear map from $\E_x$ to $\E_y$;
\item[(ii)]  $H^s_{xx}=\Id\,$ and $\,H^s_{yz} \circ H^s_{xy}=H^s_{xz}$;
% for all $x\in \M$ and $y,z \in W^s(x)$, 
\item[(iii)]  $H^s_{xy}= (F^n_y)^{-1}\circ H^s_{f^nx \,f^ny} \circ F^n_x\;$ 
for all $n\in \N$.
\end{itemize}
\end{definition}

\noindent Unstable holonomy are defined similarly.
The following proposition establishes existence 
(cf. \cite{ASV} for trivial bundles) and some additional properties of the holonomies.  
We use identifications $\I_{xy}$ defined in Section \ref{bundle} .
Holonomies do not depend on the 
choice of identification as they are unique by (c).
We denote by $W^s_{loc}(x)$ a sufficiently small ball around 
$x$ in the leaf $W^s(x)$.

\begin{proposition} \label{close to Id} 

Suppose that the cocycle $F$ is fiber bunched.
Then there exists $C>0$ 
such that for any $x\in \M$ and $y\in W^s_{loc}(x)$,
\begin{itemize}
\vskip.2cm

\item[(a)] $\|(F^n_y)^{-1} \circ \I_{f^{n}x f^{n}y} \circ F^n_x - \I_{xy} \,\| \leq C\dist (x,y)^{\beta}\,$
for every $n\in \N$;
\vskip.2cm

\item[(b)] $H^s_{xy} =\underset{n\to\infty}{\lim} (F^n_y)^{-1} \circ \I_{f^{n}x f^{n}y} \circ F^n_x\,$
%exists and 
satisfies (i), (ii), (iii) of Definition \ref{holonomies} and 
 $$ (iv)\;\| H^s_{xy} - \I_{xy} \,\| \leq C\dist (x,y)^{\beta};$$
  
  \item[(c)] The stable holonomy satisfying (iv) is unique.
  
  \end{itemize}
$H^s_{xy}$ can be extended to any $y\in W^s(x)$ using (iii).  Similarly, for  $y\in W^u(x)$, the unstable holonomy $H^u$ is obtained as 
 $H^u_{xy} =\underset{n\to-\infty}{\lim} (F^n_y)^{-1} \circ \I_{f^{n}x f^{n}y} \circ F^n_x.$ 
  \end{proposition}
 
 \begin{proof} 
 {\bf (a)}  We fix $x\in\M$ and denote $\,x_i=f^i(x)$. Then for any $y\in W^s_{loc}(x)$ we have
\begin{equation}
 \begin{aligned} \label{2.1}
 & (F^n_y)^{-1}\circ \I_{x_{n} y_{n}} \circ F^n_x \,= (F^{n-1}_y)^{-1}\circ \left( 
   (F_{y_{n-1}})^{-1} \circ \I_{x_{n} y_{n}} \circ F_{x_{n-1}}\right) \circ F^{n-1}_x =\\
  & = (F^{n-1}_y)^{-1} \circ (\I_{x_{n-1} y_{n-1}}+r_{n-1}) \circ F^{n-1}_x =\\
  & = (F^{n-1}_y)^{-1}\circ \I_{x_{n-1} y_{n-1}}\circ F^{n-1}_x+(F^{n-1}_y)^{-1}\circ r_{n-1}
  \circ F^{n-1}_x  = \dots =\\
 &=\I_{xy}+\sum_{i=0}^{n-1} (F^{i}_y)^{-1}\circ r_{i}\circ F^i_x ,
 \quad \text{where  }r_i=(F_{y_i})^{-1}\circ \I_{x_{i+1} y_{i+1}} \circ F_{x_i}-\I_{x_{i} y_{i}}. 
 \end{aligned}
\end{equation}

Since $F$ is fiber bunched, there is $\theta<1$ such that 
$\|F(x)\|\cdot \|F(x)^{-1}\| \cdot  \nu(x)^\beta <\theta$ for every $x$ in $\M$.
For the function $\nu$ we denote its trajectory product by
 $$
 \nu_i(x)=\nu(x) \nu(fx) \dots \nu (f^{i-1}x)=
\nu(x_0) \nu(x_1) \dots \nu (x_{i-1}) , \quad i\in \N.
 $$
Then one can estimate 
$\dist (f^nx, f^ny)\le  \dist (x,y) \nu_n(y)$, see e.g. \cite[Lemma 1.1]{BW}.

 %%%%%%%%%%%%%%%%%%%%%%%
\begin{lemma} \label{F_i}
There is $C_0$ such that for every $x\in \M$, $y\in W^s_{loc}(x)$, and $i\ge 0$, 
$ \;\| ( F^i_y )^{-1} \| \cdot  \|\, F^i_x\,\| \le C_0 \,\theta^i \, \nu_i (y)^{-\beta} $.  
\end{lemma}

\begin{proof}
Using \eqref{I} and \eqref{FHolder} we obtain
$$
  \frac{\|F_{x_k}\|}{\|F_{y_k}\|} \le
    \frac{\|F_{x_k}-\I_{x_{k+1} y_{k+1}}^{-1}  \circ F_{y_k} 
    \circ \I_{x_{k} y_{k}} \| } {\|F_{y_k}\|} 
     +  \frac{ \|\I_{x_{k+1} y_{k+1}}^{-1}  \circ F_{y_k} \circ 
   \I_{x_{k} y_{k}}\| }{\|F_{y_k}\|}\le
 $$
 $$   
 \le  C_1  (\dist (x_k,y_k))^\beta +   \|\I_{x_{k+1} y_{k+1}}^{-1}  \| \cdot \| \I_{x_{k} y_{k}}\|   \le 1 + C_2  (\dist (x_k,y_k))^\beta.
 $$ 
 We estimate
$$
\| ( F^i_y )^{-1} \| \cdot  \| F^i_x\|  
\le   \|(F_y)^{-1}\| \cdot \|(F_{y_1})^{-1}\|  \cdots 
     \|(F_{y_{i-1}})^{-1}\| \cdot
      \|F_x\| \cdot \|F_{x_1}\| \cdots  \|F_{x_{i-1}} \|
$$   
$$
 \hskip.5cm  \le  \prod_{k=0}^{i-1} \|F_{y_k}\| \, \|(F_{y_k})^{-1}\| \cdot  \prod_{k=0}^{i-1} \frac{\|F_{x_k}\|}{\|F_{y_k}\|} < \;
      \prod_{k=0}^{i-1}  \theta \,\nu (y_k)^{-\beta} \cdot  \prod_{k=0}^{i-1} 
      \left( 1 + C_2 \left( \dist (x_k, y_k) \right)^\b \right). 
 $$ 
Since the distance between $x_n$ and $y_n$ decreases exponentially,
the second product is uniformly bounded and we obtain   
 $\|(F^i_y )^{-1}\| \cdot \|\, F^i_x\,\|  \le C_0 \,\theta^i \,\nu_i (y)^{-\beta}$.
\end{proof}  
%%%%%%%%%%%%%

Since $F$ is H\"older continuous \eqref{FHolder} we have 
\begin{equation} \label{2.2}
\begin{aligned}
  \|r_i\| &=  \|((F_{y_i})^{-1}\circ \I_{x_{i+1} y_{i+1}} \circ F_{x_i}-\I_{x_{i} y_{i}} \| 
  \,\le\,  \|(F_{y_i})^{-1}\circ \I_{x_{i+1}}\| \cdot \\
  & \cdot \| F_{x_i} -  \I_{x_{i+1} y_{i+1}} ^{-1} \circ F_{y_i} \circ \I_{x_{i} y_{i}}\|    \le  C_3\dist(x_i, y_i)^\beta  
    \le C_3 ( C_4 \,\dist(x, y) \,\nu_i (y)  )^\beta 
  \end{aligned}
\end{equation}

%%%%%%%%%%%

It follows from \eqref{2.2} and Lemma \ref{F_i}
that for every $i\ge 0$,
\begin{equation} \label{2.3}
\begin{aligned}
& \|(F^{i}_y)^{-1}\circ r_{i}\circ F^i_x\| \le 
   \|(F^{i}_y)^{-1}\| \cdot \|F^i_x\|  \cdot \|r_{i}\| \le \\
 &  \le  C_0 \,\theta^i \, \nu_i (y)^{-\beta} \,C_3 C_4^\beta\,
  \dist(x,y)^\beta \nu_i(y)^\beta = C_5\, \dist(x,y)^\beta \, \theta^i .
 \end{aligned}
\end{equation}

 Using \eqref{2.1}, \eqref{2.3} and convergence of $\sum \theta^i$ we 
 conclude that  
$$
\begin{aligned}
  \| (F^n_y)^{-1}\circ  \I_{x_{n} y_{n}}\circ F^n_x - \I_{xy} \|  \,\le\, \sum_{i=0}^{n-1} 
  \| (F^{i}_y)^{-1}\circ r_{i}\circ F^i_x \| \le  
 C \,\dist(x,y)^\beta.
\end{aligned}
  $$

%\vskip.3cm 

 {\bf (b)}
It follows from \eqref{2.1} that 
$$
\|(F^{n+1}_y)^{-1}\circ  \I_{x_{n+1} y_{n+1}} \circ F^{n+1}_x - (F^n_y)^{-1}\circ  \I_{x_{n} y_{n}} \circ F^n_x \| =
\|(F^n_y)^{-1}\circ r_n\circ F^n_x \|.
$$
Hence  $\{(F^n_y)^{-1}\circ \I_{x_{n} y_{n}} \circ F^n_x\}$ is a Cauchy sequence by  \eqref{2.3},
and thus it has a limit $H^s_{xy}:\E_x\to\E_y.$ 
Since the convergence is uniform on the set of pairs $(x, y)$ where 
$y\in W^s_{loc}(x)$,  the map $H^s$ is continuous.
Clearly, the maps $H^s_{xy}$ are linear and satisfy $H^s_{xx}=\Id$.
It follows from (a) that $\| H^s_{xy} - \I_{xy} \,\| \leq C\dist (x,y)^{\beta}$.
We also have
$$ 
H^s_{xy}  =\underset{k\to\infty}{\lim} 
(F^n_y)^{-1} \circ (F^{k-n}_{f^ny})^{-1}\circ \I_{f^{k}x f^{k}y} \circ F^{k-n}_{f^nx}\circ F^n_x
=(F^n_y)^{-1}\circ H^s_{f^nx \,f^ny} \circ F^n_x.
$$
To show $\,H^s_{yz} \circ H^s_{xy}=H^s_{xz}$ we use \eqref{I} and
Lemma \ref{F_i}   to obtain as in \eqref{2.3} that
 $$
 \| H^s_{xz} -H^s_{yz} \circ H^s_{xy}\|  \le
 \|(F^n_z)^{-1}\| \cdot \| (\I_{x_{n} z_{n}}- \I_{y_{n} z_{n}} \circ \I_{x_{n} y_{n}} ) \| \cdot \|  F^n_x\|  \to 0 \quad\text {as }n\to \infty.
 $$

\vskip.3cm 

 {\bf (c)} Suppose that $H^1$ and $H^2$
are two stable holonomies satisfying 
 $\| H^{1,2}_{xy} - \I_{xy} \,\| \leq C\dist (x,y)^{\beta}$. Then
 using Lemma \ref{F_i} we obtain
 $$
 \begin{aligned}
 \|H^1_{xy}-H^2_{xy}\| & =
 \|(F^n_y)^{-1}\circ (H^1_{f^nx \,f^ny}- H^1_{f^nx \,f^ny}) \circ F^n_x\| \le \\
 & \le C_0 \,\theta^n \, \nu_n (y)^{-\beta} \, C\dist (f^nx,f^ny)^{\beta} \le
 C_6 \,\theta^n \to 0 \quad\text {as }n\to \infty,
 \end{aligned}
 $$
 and hence $H^1=H^2$.
\end{proof}

%%%%%%%%%%%%%%%%%%%%%
%%% Conformal structure is continuous %%%
%%%%%%%%%%%%%%%%%%%%%

\subsection{Proof of Theorem \ref{structure}}

We use the distance between conformal structures described in 
Section \ref{conf structure} and the identification $\I$ from
Section \ref{bundle}. We note that the proof works without change
under the assumption that $\E$ and $F$ are continuous 
and that $F$ has continuous holonomies.

Let $\tau$ be an $F$-invariant $\mu$-measurable conformal structure
on $\E$. We first show that $\tau$ is essentially invariant under the stable
 and unstable holonomies of $F$.

\begin{proposition} \label{tauHinvariant}
Suppose that $H^s$ is a stable holonomy for a linear cocycle $F$. 
If $\tau$ is a measurable  $F$-invariant conformal structure then  
$\tau$ is essentially $H^s$-invariant, i.e.  there is a set $G\subset \M$ 
of full measure such that 
 $$
\tau(y)= H^s_{xy}(\tau(x)) \quad \text{for all }x, y \in G 
\;\text{ such that }y\in W^s_{loc}(x). 
$$ 

\end{proposition}

\begin{proof}
We denote by $F^n_x (\xi)$ the push forward of a conformal 
structures $\xi$ from $\E (x)$ to $\E(f^nx)$ induced by $F^n_x$,
and similar notations for push forwards by $H^s$ and $\I$.
We also let $x_i=f^i(x)$. Since  $\tau$ is $F$-invariant, 
and $F^n_y$ induces an isometry, we obtain
$$
  \begin{aligned}
&\dist\,(\tau(y), H^s_{xy} (\tau(x))) = 
\dist (F^n_y (\tau(y)), F^n_y H^s_{xy}(\tau(x))) = \\
& = \dist (\tau(y_n), H^s_{x_ny_n} F_x^n(\tau(x)))
 =\dist( \tau(y_n), H^s_{x_ny_n} (\tau(x_n)))\le\\
&\le  \dist( \tau(y_n), \I_{x_{n} y_{n}} (\tau(x_n))) + 
\dist( \I_{x_{n} y_{n}} (\tau(x_n)), H^s_{x_ny_n} (\tau(x_n))). 
 \end{aligned}
 $$

Since  $\tau$ is $\mu$-measurable, 
by Lusin's Theorem there exists a compact set $S\subset \M$ 
with $\mu(S)>1/2$ on which $\tau$ is uniformly continuous and 
hence bounded. 
Let $G$ be the set of points in $\M$ for which the frequency of 
 visiting $S$ equals $\mu(S)>1/2$.
By Birkhoff Ergodic Theorem, $\mu(G)=1$.

Suppose that  both $x$ and $y$ are 
 in $G$. Then there exists a sequence $\{ n_i \}$ such that 
 $x_{n_i}\in S$ and $y_{n_i}\in S$.
Since $y\in W^s_{loc}(x)$, $\,\dist(x_{n_i}, y_{n_i})\to 0$ and hence 
$\dist(\I_{x_{n_i} y_{n_i}}(\tau(x_{n_i})), \tau(y_{n_i}))\to 0$ by uniform continuity of $\tau$ on $S$. 
Since $H^s$ and $\I$ are continuous  and satisfy $H^s_{xx}=\Id=\I_{xx}$,
we have $\|\I_{x_{n_i} y_{n_i}}^{-1} \circ H^s_{x_{n_i}y_{n_i}} - \Id \| \to 0$. 
Since $\tau$ is bounded on $S$, the lemma below yields
$$
\dist( \I_{x_{n_i} y_{n_i}}(\tau(x_{n_i})), H^s_{x_{n_i}y_{n_i}} (\tau(x_{n_i})))=
\dist(\tau(x_{n_i}), \, \I_{x_{n_i} y_{n_i}}^{-1} \circ H^s_{x_{n_i}y_{n_i}} (\tau(x_{n_i}))) \to 0.
$$
 We conclude that $\dist\,(\tau(y), H^s_{xy} (\tau(x))) = 0$ and thus
$\tau$ is essentially $H^s$-invariant.

\begin{lemma}\label{A*} \cite[Lemma 4.5]{KS10}
Let $\sigma$ be a conformal structure on $\R^d$ and $A$ be a linear
transformation of $\,\R^d$ sufficiently close to
the identity. Then 
  $$
     \dist\,(\sigma, A(\sigma)) \le k(\sigma)\cdot\|A-\Id\,\|,
  $$
where $k(\sigma)$ is bounded on compact sets in $\c^d$. 
More precisely, if $\sigma$ is given by a matrix $C$, then
$k(\sigma) \le 3d \, \|C^{-1}\|\cdot \|C\|$
 for any $A$ with $\|A-\Id\,\| \le (6 \|C^{-1}\|\cdot \|C\|)^{-1}$. 

\end{lemma}
\end{proof}

Similarly, $\tau$ is essentially $H^u$-invariant. Since the stable
and unstable holonomies of $F$ are continuous we conclude that 
$\tau$ is essentially uniformly continuous along $W^s$ and $W^u$. 
Since the base system $f$ is center bunched and accessible this 
implies continuity of $\tau$ on $\M$ by \cite[Theorem E]{ASV} or 
\cite[Theorem 4.2]{W}.

$\QED$

%%%%%%%%%%%%%%%%%%%%%%%%%%

%%%%%%%%%%%%%%%%%%%%%%%%%

\subsection{Proof of Proposition \ref{qc implies conf}}
We use the following proposition from \cite{KS10}.
Recall that a measurable conformal structure $\tau$ on 
$\E$ is called {\em bounded} 
if the distance between $\tau(x)$ and $\tau_0(x)$ is essentially
bounded on $\M$ for a continuous conformal structure $\tau_0$ 
on $\E$.    

\begin{proposition} \cite[Proposition 2.4]{KS10} \label{bounded measurable} 
Let $f$ be a homeomorphism of a compact manifold $\M$ and let
$F : \E \to \E$ be a continuous linear cocycle over $f$. If $F$ is uniformly 
quasiconformal then it  preserves a bounded measurable conformal 
 structure $\tau$ on $\E$.
\end{proposition}

Under our standing assumptions,  Theorem \ref{structure} now
implies that  $\tau$ is  continuous. We can normalize it by a  continuous
function on $\M$ to obtain a Riemannian metric with respect to which
$F$ is conformal.
$\QED$

%%%%%%%%%%%%%%%%%%%%%%%%%%%%
%%%%%%%%%%%%%%%%%%%%%%%%%%%%

\subsection{Proof of Theorem \ref{distribution}} 
Let $\E'$ be a measurable $F$-invariant sub-bundle of $\E$
with $\dim \E'_x=d'$.
We consider a fiber bundle $\G$ over $\M$ whose fiber over $x$ 
is the Grassman manifold $\G _x$ of all $d'$-dimensional 
subspaces in $\E_x$. Then $F$ induces the cocycle $\F : \G \to \G$ 
over $f$ with diffeomorphisms $\F_x : \G_x \to \G_{fx}$ depending
continuously on $x$ in smooth topology. 

The stable holonomy $H^s$ for $F$ induces a stable  holonomy 
$\tilde H^s$ for $\F$. Similarly, to the linear case, this is a family of diffeomorphisms 
$\tilde H^s_{xy} : \G _x \to \G _y$ that satisfies properties (ii) and (iii)
Definition \ref{holonomies}
and depends continuously on $x$ and $y \in W^s_{loc}(x)$.
Similarly  $H^u$ induces the  unstable holonomy $\tilde H^u$ for $\F$.

The   sub-bundle  $\E'$ gives rise to a $\mu$-measurable 
$\F$-invariant section $\phi:\M \to \G$. 
We take $m$ to be the lift of $\mu$ to the graph $\Phi$ of $\phi$, i.e.
for a set $X\subset \G$ we define $\;m(X)=\mu(\pi(X\cap \Phi))$, where 
$\pi : \G \to \M$ is the projection. Equivalently, $m$ can be
defined by specifying that  for $\mu$-almost every $x$ in $\M$ the conditional 
measure $m_x$ in the fiber $\G _x$ is the atomic measure at $\phi(x)$.
Since $\mu$ is $f$-invariant and $\Phi$ is $\F$-invariant,
the measure $m$ is $\F$-invariant.

\begin{lemma} \cite[Lemma 4.6]{KS10}\label{neutral_lemma}
There exists $C>0$ such that for any $x \in \M$,  
subspaces $\xi, \eta \in \G_x$, and $n \in \Z$ 
\begin{equation} \label{neutral} 
\dist (\F^n_x (\xi), \F^n_x (\eta)) \le C\cdot K_F(x,n) \cdot \dist (\xi, \eta)
\end{equation}
\end{lemma}

The definitions of $K(x,n)$,  $\lambda_+(F,\mu )$, and 
$\lambda_-(F,\mu )$   yield  that for $\mu$ almost all $x$
\begin{equation} \label{K,lambda}
\begin{aligned}
&\lim _{n\to \infty} \frac 1n \log {K(x,n)} =
\lim _{n\to \infty} \frac 1n \log (\|F^n_x\|\cdot \|(F^n_x)^{-1}\|) = \\
&\lim _{n\to \infty} \frac 1n \log \|F^n_x\| -
\lim _{n\to \infty} \frac 1n \log \|(F^n_x)^{-1}\|^{-1} =
\lambda_+(F,\mu )-\lambda_-(F,\mu )=0.
\end{aligned}
\end{equation}
Hence Lemma \ref{neutral_lemma} implies that  Lyapunov
exponent of $\F$ along the fiber is zero $m$ a.e.
This together with existence of the stable and unstable 
 holonomies for $\F$ allows us to apply 
\cite[Theorem C]{ASV}
to the measure $m$ and conclude that there exists a system of 
conditional measures $\tilde m_x$ on $\G_x$ for $m$ which are 
holonomy invariant and depend continuously on $x\in \M$
in the weak$^*$ topology. 

Since the conditional measures $m_x$ and 
$\tilde m_x$ coincide for all $x$ in  a set $X\subset \M$ 
of full $\mu$ measure, we see that
$\tilde m_x = m_x$ is the atomic measure at $ \phi(x)$ for all $x\in X$. 
Since $X$ is dense we obtain that $\tilde m_x$ is atomic for all $x\in \M$. 
Indeed,  for any $x\in \M$  we can take a sequence $X \ni x_i  \to x$  
and assume by compactness of $\G$ that $\phi (x_i)$ converge 
to some $\xi \in \G_x$. This implies that $\tilde m_{x_i}=m_{x_i}$ converge to 
the atomic measure at $\xi$, which therefore coincides with $\tilde m_x$ 
by continuity of the family $\{\tilde m_x\}$. 
Denoting $\tilde \phi(x) = \text{supp } \tilde m_x$ for $x\in \M$, 
we obtain a continuous section $\tilde \phi$ which  coincides with 
$\phi$ on $X$. 
This shows that $\E'$ coincides $\mu$-almost everywhere with a  
continuous sub-bundle
which is invariant under the stable and unstable holonomies. 
$\QED$

%\vskip1cm
%%%%%%%%%%%%%%%%%%%%%%%%%%%%%%%%%%%

%\newpage

\subsection{Proof of Theorem \ref{reduction}}

We use the following particular case of Zimmer's Amenable Reduction Theorem:
\vskip.2cm

 \noindent 
 \cite[Corollary 1.8]{HKt},  \cite[Theorem  3.5.9]{BP}
 {\it Let  $f$ be an ergodic transformation of a measure space 
 $(X,\mu)$ and let $F: X \to GL(d,\R)$ be a measurable function.
 Then there exists a measurable function $C :X \to GL(d,\R)$ 
 such that the function $A(x) = C^{-1}(fx) F(x) C(x)$ takes
 values in an amenable subgroup of $GL(d,\R)$.}

\vskip.2cm

%Amenable subgroups of Lie groups were studied in \cite{Moore}.
There are $2^{d-1}$ standard maximal amenable subgroups of $GL(d,\R)$.
They correspond to the distinct compositions of $d$, $\, d_1+\dots + d_k=d$,
and each group consists of all block-triangular matrices of the form
\begin{equation} \label{blockform}
 \left[ \begin{array}{cccc}
 A_1 & \ast & \ldots & \ast \\
 0  & A_2 & \ddots &\vdots \\
 \vdots & \ddots & \ddots & \ast \\
 0 & \ldots & 0 & A_k
 \end{array} \right]
\end{equation}
 where each  diagonal block $A_i$ is a scalar
 multiple of a $d_i\times d_i$ orthogonal matrix.   Any amenable
 subgroup of $GL(d,\R)$ has a finite index subgroup which is
 contained in a conjugate of one of these standard subgroups 
 \cite[Theorem 3.4]{Moore}. The normalizer of the diagonal 
subgroup is an example of an amenable group which does not lie 
in any such conjugate. It is the finite extension of the diagonal 
subgroup that contains all permutations of the coordinate axes.

Since $\E$ can be trivialized on a set of full $\mu$-measure
\cite[Proposition 2.1.2]{BP},  we can measurably identify $\E$ with 
$\M \times \rd$ and  view $F$ as a function $\M \to GL(d,\R)$. 
Thus can we apply the Amenable 
Reduction Theorem to $F$ and obtain a measurable coordinate change 
function $C :\M \to GL(d,\R)$ such that $A(x) = C^{-1}(fx) F(x) \,C(x) \in G$ 
for $\mu$-a.e. $x\in \M$, where $G$ is an amenable subgroup of 
$GL(d,\R)$. By the above we may assume that $G$ contains a 
finite index subgroup $G_0$ which is contained in one of the $2^{d-1}$
standard maximal amenable subgroups. 

{\bf Case 1.} First we consider the case when $G$ itself is contained in a
standard subgroup. Then the conclusion of the theorem
holds for $F$ itself rather than for a power of its lift.
The sub-bundle $V^i$ spanned by the first $d_1+\dots + d_i$
coordinate vectors in $\R^d$ is $A$-invariant for $i=1, \dots, k$.
Denoting $\E^i_x=C(x) V^i$ we obtain  the corresponding  
flag of measurable $F$-invariant  sub-bundles 
$$
\E^1\subset \E^2 \subset \dots \subset \E^k=\E \quad\text{with }\;
\dim \E^i=d_1+\dots + d_i.
$$
By Theorem \ref{distribution} we may assume that the sub-bundles $\E^i$ are continuous.
Since $A_1(x)$ is a scalar multiple of a $d_1\times d_1$ orthogonal matrix
for $\mu$-a.e. $x$, we conclude that the restriction of $F$ to $\E^1$ 
is conformal with respect to the push forward by $C$ of the standard 
conformal structure on $V^1$. This gives a measurable $F$-invariant 
conformal structure $\tau_1$ on $\E^1$. Since $F$ preserves $\E^1$, 
so does the stable holonomy $H^s$. Hence $H^s$ induces 
a stable holonomy $H^{s,1}$ for the restriction of $F$ to $\E^1$. 
By Proposition \ref{tauHinvariant} the conformal structure $\tau_1$
is essentially
invariant under $H^{s,1}$. Similarly we obtain essential invariance 
of $\tau_1$ under the unstable holonomy $H^{u,1}$. This yields 
continuity of $\tau_1$ on $\M$ as in the end of the proof of 
Theorem \ref{structure}. 

Similarly, we can consider continuous factor-bundle $\E^i/\E^{i-1}$
over $\M$ with the natural induced cocycle $F^{(i)}$. Since the
matrix of the map induced by $A$ on $V^i/V^{i-1}=\R^{d_i}$ is $A_i$,
it preserves the standard conformal structure on $\R^{d_i}$. Pushing 
forward by $C$ we obtain a measurable conformal structure 
$\tau_i$ on $\E^i/\E^{i-1}$ invariant under $F^{(i)}$. The holonomies 
$H^s$ and $H^u$ 
induce continuous holonomies for $F^{(i)}$ on $\E^i/\E^{i-1}$. 
As above, we conclude that $\tau_i$ is essentially invariant under
these holonomies and continuous. 

{\bf Case 2.} Now we consider the case when only a finite
 index subgroup $G_0$ of $G$ is contained in a standard maximal
amenable  subgroup.  We again consider the flag of subspaces 
$V^i$ spanned by the first $n_i=d_1+\dots + d_i$ coordinate vectors 
in $\R^d$, $i=1, \dots, k$. Let $G_*$ be the stabilizer of this flag in $G$.
Then $G_*$ contains $G_0$ and thus $G_*$ has a finite finite index $l$
in $G$. The orbit of this flag under $G$ consists of $l$ 
distinct flags in $\R^d$ which we denote by 
\begin{equation} \label{flagsW^j}
W^j= \{ V^{j,1} \subset ... \subset V^{j,k-1} \subset V^{j,k}=\R^d \}, \qquad j=1,...,l.
\end{equation}
Any $g\in G$ permutes these flags and preserves their union.
First we will construct  
corresponding flags of continuous invariant sub-bundles. If $l \ge 2$ this
requires in general to pass to a finite cover and a power of $F$.
After that we will show existence of continuous invariant conformal 
structures on the factor bundles for each flag. 

For each $i=1, ..., k-1$ the subspaces $V^{j,i}$, $j=1,...,l$, have 
dimension $n_i$. Some of them may coincide, so we denote the
number of distinct ones by $l_i$. We also denote their union and its 
image under $C$ by 

\begin{equation} \label{UhatU}
U^{(i)}= V^{1,i} \cup ... \cup V^{l,i} \subset \R^d \quad \text{and} 
\quad \hat \U^{(i)} _x=C(x) U^{(i)} \subset \E_x \quad \text{for } \mu\text{-a.e. } x
\end{equation}
Then $\hat \U^{(i)}$ depends measurably on $x$ and is a union of 
${l_i}$ distinct $n_i$-dimensional subspaces of $\E_x$ . Since $g(U^{(i)})=U^{(i)}$ for any 
$g\in G$, we see that $\hat \U^{(i)} _x$ is invariant under $F$, i.e.  
$F_x (\hat \U^{(i)} _x) = \hat \U^{(i)} _{fx}$. 
First we claim that for each $x$ in $\M$ there exists $\, \U^{(i)} _x$, a union 
of ${l_i}$ distinct $n_i$-dimensional subspaces of $\E_x$,  
which depends continuously on $x$, coincides with
$\hat \U^{(i)} _x$ for $\mu$-a.e. $x$, and is invariant under $F$.
This can be seen as in the proof of Theorem \ref{distribution}. Indeed,
we can define the measure $m$ on the corresponding Grassmannian 
bundle $\G$ by choosing the conditional measure $\hat m_x$ on $\G_x$ to 
be the atomic measure equidistributed on the ${l_i}$ points 
corresponding to $\hat \U^{(i)} _x$. 
Then $m$ is invariant under the induced cocycle  on $\G$.
Thus by the same argument we obtain that there exists a system of 
conditional measures $ m_x$ on $\G_x$ for $m$ which are 
holonomy invariant and depend continuously on $x\in M$
in the weak$^*$ topology. As before, the measures $ m_x$ are
atomic for all $x$. Moreover, the number of atoms is preserved by 
the stable and unstable holonomies since they are induced by 
linear isomorphisms. Hence accessibility implies that the number 
of atoms is ${l_i}$ for all $x$ in $\M$. Then $m_x$ corresponds to
a union $\U^{(i)} _x$ of $l_i$ distinct $n_i$-dimensional subspaces 
in $\E _x$ which depends continuously 
on $x$ and is invariant under $F$. We also denote 
 $\U^{(i)} = \bigcup _{x \in \M} \U^{(i)} _x \subset \E.$

We fix a point $q \in \M$ and denote the subspaces in $ \U^{(i)} _q$
by $\U^{1,i}_q , ... , \,\U^{l_i,i}_q$. Locally they can be uniquely extended 
to continuous sub-bundles $\U^{1,i}, ... , \,\U^{l_i,i}$ so that the union
of $\U^{1,i}_x , ... , \,\U^{l_i,i}_x$ is $\U^{(i)} _x$. Similarly, they 
can be extended uniquely along any curve. The extension along 
 a loop based at $q$ produces a permutation of subspaces 
 $\U^{1,i}_q , ... , \,\U^{l_i,i}_q$, which depends only on the homotopy 
 class of the loop. Thus we obtain a natural  homomorphism 
 $\rho_i :\pi_1(\M, q) \to \Sigma({l_i})$ from the fundamental  group of 
 $\M$ to the  group of permutations of $l_i$ symbols. If the homomorphism 
 $\rho_i$ is trivial, then $\U^{1,i}_q , ... , \,\U^{l_i,i}_q $ extend globally
to continuous sub-bundles $\U^{1,i}, ... , \,\U^{l_i,i}$ over $\M$. 
 
We consider the homomorphisms $\rho_i$ for each $i=1, ..., k-1$ and 
denote their direct product by 
$\rho :\pi_1(\M, q) \to \Sigma({l_1}) \times ... \times \Sigma({l_{k-1}})$. 
If $\rho$ is non-trivial
we pass to a finite cover as follows. The kernel $H_q=\ker \rho$ is a 
normal subgroup of finite index in $\pi_1(\M,q)$.  Hence there exists 
a finite cover 
$p: \tilde \M \to \M$ such that $p_* (\pi_1(\tilde \M, \tilde q))=H_q$ 
for any $\tilde q \in p^{-1}(q)$. By taking the pullback under $p$ we obtain
the bundle $\tilde \E$ over $\tilde \M$ and for each $i=1, ..., k-1$ the corresponding  union of subspaces $\tilde \U^{(i)} _{\tilde x}$ in $\tilde \E _{\tilde x}$. 
Fix some $\tilde q \in p^{-1}(q)$.
We claim that for each $i$ the subspaces in $\tilde \U^{(i)}_{\tilde q} $ 
extend to continuous $n_i$-dimensional sub-bundles 
$\tilde \U^{1,i}, ... , \,\tilde \U^{l_i,i}$ of $\tilde  \E$ so that 
$\tilde \U^{1,i}_{\tilde x} \cup ... \cup \,\tilde \U^{l_i,i}_{\tilde x}=\tilde \U^{(i)} _{\tilde x}$ 
for all $\tilde x \in \tilde \M$. It suffices to check that any loop 
$\tilde \gamma \in \pi_1(\tilde \M, \tilde q)$ induces trivial permutations.
Indeed, $p\circ \tilde \gamma \in H_q$ by the construction of $\tilde \M$, 
and $\tilde \gamma$ and $p \circ \tilde \gamma$ induce the same 
permutations since the extension  along $\tilde \gamma$ projects to 
that along $p \circ \tilde \gamma$. 

Now we lift $f$ to the cover $\tilde \M$. Fix any $\tilde q \in p^{-1}(q)$ and choose any $\tilde x  \in p^{-1}(f(q))$. 
$$
\begin{array}{ccccccc}
\tilde \M, \,\tilde q &\overset{\text{\small{$\tilde f$}}}{\longrightarrow} & \tilde \M,\, \tilde x &\qquad\qquad&
\pi_1(\tilde \M, \,\tilde q) & & 
\pi_1(\tilde \M, \tilde x )
\\
\text{\small{$p$}}\downarrow & &\text{\small{$p$}}\downarrow &&
\text{\small{$p_*$}}\downarrow & &\text{\small{$p_*$}}\downarrow 
 \\
\M, \, q & \overset{\text{\small{$f$}}}{\longrightarrow}&  \M,\, f(q) &&
\pi_1(\M, q) & \overset{\text{\small{$f_*$}}}{\longrightarrow}&  \pi_1(\M, f(q))
\end{array}
$$
A necessary and sufficient condition for existence of 
$\tilde f : \tilde \M \to \tilde\M$ satisfying $\tilde f (\tilde q) = \tilde x$ 
and covering $f \circ p$ is the inclusion 
$(f \circ p)_*  (\pi_1(\tilde \M, \tilde q)) \subset p_* (\pi_1(\tilde \M, \tilde x))$. 
We will show the equality. 
We recall that $p_*  (\pi_1(\tilde \M, \tilde q)) = H_q$ by the construction.
It follows that $p_*  (\pi_1(\tilde \M, \tilde x))$ is the subgroup $H_{f(q)}$ of 
$\pi_1(\M,f(q))$ that consists of all loops inducing  trivial permutations
of the subspaces at $f(q)$. Indeed, for any natural isomorphism 
$i_s : \pi_1(\M,f(q)) \to \pi_1(\M,q)$ given by a path $s$ from $q$ to $f(q)$
it is easily seen that $i_s (H_{f(q)})=H_q$, and for any cover
$i_s (p_*  (\pi_1(\tilde \M, \tilde x)))=p_*  (\pi_1(\tilde \M, \tilde q))$ 
since the latter
is normal. Finally, $f_*(H_q) = H_{f(q)}$ since for any loop 
$\gamma \in \pi_1(\M,q)$ the cocycle $F$ gives a homeomorphism
between the restrictions of $\E$ to $\gamma$ and to $f \circ \gamma$, 
which maps $\U^{(i)}$ to $\U^{(i)}$ and hence preserves the type of the induced permutation. We conclude that $(f \circ p)_*  (\pi_1(\tilde \M, \tilde q)) =
 p_* (\pi_1(\tilde \M, \tilde x))$ and thus the lift $\tilde f$  exists. 
\vskip.1cm
 
We note that the manifold $\tilde \M$ is compact and connected, and 
the lift $\tilde f$ 
satisfies our standing assumptions. Indeed, since the projection $p$ is 
a local diffeomorphism, the invariant volume $\mu$ lifts to an invariant 
volume $\tilde \mu$, and the partially hyperbolic splitting for $f$ lifts to the one 
for $\tilde f$. Moreover, $\tilde f$ satisfies the same bunching and its stable/unstable foliations project to those of $f$. It follows that $\tilde f$ 
is accessible. Indeed,
by compactness and connectedness, it suffices to show that
any point $\tilde q \in \tilde \M$ has a neighborhood whose any point 
$\tilde y$  can be connected 
to $\tilde q$ by an $su$-path. Let $q=p(\tilde q)$ and $y=p(\tilde y)$.
By \cite[Lemma 4.4]{W}  $q$ can be connected to any sufficiently close
$y$ by an $su$-path arbitrarily close to a certain contractible $su$-path 
from $q$ to $q$. The lift of such a path
to $\tilde \M$ gives an $su$-paths connecting  $\tilde q$ and $\tilde y$.
 
We  denote by $\tilde F : \tilde  \E \to \tilde  \E$ the unique lift of the cocycle 
$F$ to the cocycle over $\tilde f$. We note that $\tilde  \E$ and $\tilde F$ 
are $\beta$-H\"older, and $\tilde F$ is fiber-bunched. We also lift the matrix
functions $A$ and $C$ to $\tilde \M$, $\tilde A(\tilde x) = A(p(\tilde x))$ and 
$\tilde C(\tilde x) = C(p(\tilde x))$, and note that 
$\tilde A(\tilde x) = \tilde C^{-1}(\tilde f(\tilde x)) \tilde F(\tilde x) \,\tilde C(\tilde x)$ for 
$\tilde \mu$-a.e. $\tilde x$, where $\tilde \mu$ is the lift of $\mu$.

For each $i$ the cocycle $\tilde F$ preserves the union 
$\tilde \U^{(i)}$ of sub-bundles $\tilde \U^{1,i}, ... , \,\tilde \U^{l_i,i}$. 
Since the sub-bundles are continuous, the permutation of their order 
induced by $\tilde F_{\tilde x}$ is continuous in $\tilde x$ and hence
 is constant on $\tilde \M$. 
Hence there exists $N$ such that the cocycle $\hat F = \tilde F^{N}$
preserves every sub-bundle $\tilde \U^{j,i}$, $i=1, ..., k-1$, $j=1, ..., l_i$. Moreover, these sub-bundles  can be arranged into flags 
$\tilde \W^1, ..., \tilde \W^l$ which are mapped by $\tilde C^{-1}(\tilde x)$ 
to the flags $W^1, ..., W^l$ in $\R^d$ for $\tilde \mu$-a.e. $\tilde x$ up to 
a permutation which depends on $\tilde x$. ({We can not expect this 
permutation to be constant a.e. since 
modifying $C$ on a set of positive measure by an element of $G$
gives a different version of $A$ satisfying the conclusion of the
Amenable Reduction Theorem.})
 Indeed, for each $i$ and for $\tilde \mu$-a.e. $\tilde x$ the 
isomorphism $\tilde C(\tilde x)$ maps $U^{(i)}$ to $\tilde \U^{(i)} _{\tilde x}$ 
and thus marks the inclusions between subspaces 
$\tilde \U^{1,i-1}_{\tilde x}, ... , \,\tilde \U^{l_{i-1},i-1}_{\tilde x}$ and 
$\tilde \U^{1,i}_{\tilde x}, ... , \,\tilde \U^{l_i,i}_{\tilde x}$ that correspond to the
inclusions in the flags $W^1, ..., W^l$. We can view these inclusions
as a measurable function $\psi_i$ from $\tilde \M$ to the set of binary 
relations between the sets $\{1, ..., l_{i-1} \}$ and $\{1, ..., l_{i} \}$. 
We note that the cocycle $\hat F$ is conjugate by $\tilde C$ to the 
matrix cocycle $\hat A= \tilde A^{N}$ with values in the same group $G$. 
All elements of $G$ permute the flags $W^1, ..., W^l$ and thus preserve
the corresponding binary relations. Since $\hat F$ also preserves 
all sub-bundles $\tilde \U^{j,i}$, we conclude $\psi_i$ is invariant under
$\tilde f^N$ and hence is constant $\tilde \mu$-a.e. by ergodicity.
The inclusions 
given by the functions $\psi_i$ arrange the sub-bundles $\tilde \U^{j,i}$
into desired flags $\tilde \W^1, ..., \tilde \W^l$.
We fix one these flags and denote it 
$$
\tilde \W= \left\{ \{ 0\}=\tilde  \E^0\subset \tilde \E^1 \subset 
\dots \subset \tilde \E^k=\tilde \E \right\}.
$$
It remains to show that each factor bundle $\tilde \E^i / \tilde \E^{i-1}$, $i=1,...,k$, 
has a continuous conformal structure invariant under the factor cocycle
induced by $\hat F$.

Recall that any matrix $A$ in $G_0$ has the block-triangular form 
\eqref{blockform} and hence preserves the subspaces $V^{i}$ and $V^{i-1}$.
Its factor map on $V^{i}/V^{i-1}= \R^{d_i}$ is given by the block $A_i$ and
thus preserves the standard conformal structure $\sigma$ on $\R^{d_i}$. 
The orbit of $\sigma$ under the flag stabilizer $G_*$ is a finite set $\O$ 
in the space of conformal structures on $\R^{d_i}$. Since $\O$ is invariant
under $G_*$, so is the smallest ball containing $\O$, which is unique 
in the space of nonpositive curvature. The center of this ball is a conformal 
structure $\sigma_*$ on $V^{i}/V^{i-1}= \R^{d_i}$ preserved by the factor 
action of any matrix in $G_*$.  Pushing $\sigma_*$ by the action 
 of $G$ we obtain conformal structures $\sigma_j$ on the factors  
 $V^{j,i}/V^{j,i-1}$ of the flags $W^j$, $j=1,..., l$. It follows that if 
 $g \in G$ maps $W^j$ to $W^{j'}$ then its factor map 
 $\bar g : V^{j,i}/V^{j,i-1} \to V^{j',i}/V^{j',i-1}$ takes $\sigma_j$ to  $\sigma_{j'}$. By the construction of the flag $\tilde \W$, for $\tilde \mu$-a.e. $\tilde x$ 
 there is $j=j(x)$ 
such that $\tilde \E^i _{\tilde x}= \tilde C(\tilde x) V^{j,i}$ and $\tilde \E^{i-1} _{\tilde x}= \tilde C(\tilde x) V^{j,i-1}$.
Pushing $\sigma_j$ by $\tilde  C (\tilde x)$ 
%to a conformal structure $\tau (\tilde x)$ on $(\tilde \E^i / \tilde \E^{i-1} )_{\tilde x}$
 we obtain a measurable conformal structure $\tau$ on $\tilde \E^i / \tilde \E^{i-1}$. 
Since $\hat F$ is conjugate by $\tilde C$ to $\hat A$, which takes values 
in $G$, we conclude that $\tau$ is invariant under the factor of $\hat F$ 
on $\tilde \E^i / \tilde \E^{i-1}$. The stable and unstable holonomies for the fiber-bunched
cocycle $\tilde F$ give the continuous holonomies for $\hat F$ and 
for its factor on  $\tilde \E^i /\tilde  \E^{i-1}$. Using  
Proposition \ref{tauHinvariant} and Theorem \ref{structure} as before 
we conclude  that  $\tau$ coincides $\tilde \mu$-a.e. with a 
continuous conformal structure on $\tilde \E^i / \tilde \E^{i-1}$
invariant under the factor cocycle of $\hat F$.
 $\QED$

\begin{remark} \label{triv} {\em
Suppose that there exist $d$ continuous vector fields such that 
$\tilde \E^i$ is spanned by the first $d_1+\dots + d_i$ of them. Then 
the theorem implies that $\hat F$ is continuously cohomologous to a cocycle 
with values in a standard maximal amenable subgroup of $GL(d,\R)$ 
given by \eqref{blockform}. Indeed, since $\tilde \E \approx \M \times \R^d$ 
we can view the cocycle as a function $\hat F: \M \to GL(d,\R)$, moreover, 
the vector fields give a continuous coordinate change $\bar  C :\M \to GL(d,\R)$
 such that $\bar  A(x) = \bar  C^{-1}(fx) \hat F(x) \,\bar  C(x)$ has a block 
 triangular form. Each diagonal block $\bar A_i$ corresponds to the factor cocycle on $\E^i/\E^{i-1}$ and thus  preserves a continuous conformal 
 structure $\tau_i$ on $\R^{d_i}$, i.e. $\bar A_i(x)(\tau_i (x))=\tau_i (\tilde fx)$. 
To make the diagonal blocks conformal we change the coordinates in
$\R^{d_i}$  by the unique positive square root of the matrix of $\tau_i (x)$. 

%We note that this result does not hold without assuming existence of the vector fields  $X_i$. Even $\E=\T^2 \times \R^2$ can have a non-orientable  invariant sub-bundle $\E_1$, which makes continuous conjugacy to a triangular cocycle impossible  \cite{S11}.
}
\end{remark}
%%%%%%%%%%%%%%%%%%%%%%%%%%%%
%\newpage

\subsection{Example} \label{example}
We give an example of an analytic cocycle $F$ on $\E=\T^2 \times \R^2$
over an Anosov automorphism $f$ of $\T^2=\R^2 / \Z^2$ so that $F$
is fiber bunched and has only one Lyapunov exponent with respect to the 
Haar measure $\mu$, but has no invariant $\mu$-measurable sub-bundles 
or conformal structures. 

We view cocycles as $GL(2,\R)$-valued functions. We construct $F$ using its 4-cover
$\bar F$ on $\bar \T^2=\R^2 / (4\Z \times \Z)$, which is conjugate to a 
diagonal cocycle $\bar A$. 
%The lift $\tilde F$ on the intermediate cover $\tilde \T^2=\R^2 / (2\Z \times \Z)$ illustrates the structure we obtain in the proof of Theorem \ref{reduction}
%We consider two covers $\tilde \T^2=\R^2 / (2\Z \times \Z)$ and 
%$\bar \T^2=\R^2 / (4\Z \times \Z)$ of the standard torus.
%Let $\,B= ${\small $\left[ \begin{array}{cc}  41 & 32 \\  32 & 25 \end{array} \right]$}, or any hyperbolic  matrix in $SL(2,\Z)\,$ congruent to $\Id$ mod 4. The map $B:\R^2 \to \R^2$ projects to Anosov automorphisms $f$, $\tilde f$, and $\bar f$ of $\T^2$, $\tilde \T^2$, and $\bar \T^2$, respectively.
We define %$GL(2,\R)$-valued functions $\bar A(x)$ and $\bar C(x)$:
$$
   \bar A(x)=  \left[ \begin{array}{cc}  a(x) & 0 \\  0 & b(x)\end{array} \right],
        \; \text{ where }\;  a(x)=1+\epsilon \cos (\pi x_1), \;\;b(x)=1-\epsilon \cos (\pi x_1);
$$      
$$   \bar C(x)=    \left[ \begin{array}{cc} 
     \cos (\frac{\pi}{2} x_1) & -\sin (\frac{\pi}{2} x_1)  \\ 
      \sin (\frac{\pi}{2} x_1)  &   \;\;\; \cos (\frac{\pi}{2} x_1) \end{array} \right]=
      R \left(\frac{\pi}{2} x_1 \right), 
      \;\;\text{ the rotation by } \;\frac{\pi}{2} x_1.
$$
Both $\bar C$ and $\bar A$ are well-defined and analytic on $\bar \T^2$.
We choose a hyperbolic  matrix in $SL(2,\Z)\,$ congruent to the identity 
modulo 4,
e.g. {\tiny $\left[ \begin{array}{cc}  41 & 32 \\  32 & 25 \end{array} \right]$},
and let $f$ and $\bar f$ be the induced automorphisms of $\T^2$ and 
$\bar \T^2$, respectively. Then we define the function
\begin{equation}\label{F}
   \bar F(x)=\bar C(\bar fx)\, \bar A(x)\, \bar C(x)^{-1} 
   =\bar C(\bar fx)\, \bar C(x)^{-1} \cdot \, \bar C(x) \, \bar A(x)\, \bar C(x)^{-1}. 
 \end{equation}
The term $\bar C(\bar fx)\, \bar C(x)^{-1}$ is the rotation 
$R \left(\frac{\pi}{2} ((\bar f x)_1 -x_1 ) \right)$ and hence is 
1-periodic in both $x_1$ and $x_2$ by the assumption on $\bar f$. 
A direct calculation shows that 
$$
   \bar C(x) \, \bar A(x)\, \bar C(x)^{-1}=    \frac12 \left[ \begin{array}{cc}  
   2+\e(1+\cos(2\pi x_1))&\e\sin (2\pi x_1)
   \\ \e\sin (2\pi x_1) & 2-\e(1+\cos(2\pi x_1)) \end{array} \right].
 $$
%$$
 %  \bar C(x) \, A(x)\, \bar C(x)^{-1}=   
 %  \frac12 \left[ \begin{array}{cc} d_1(x)&c(x)\\c(x) &d_2(x) \end{array} \right], \quad
 %  \text {where} 
 %$$
%$$  
%\begin{aligned}
%     d_{1,2}(x)&= (a(x)+b(x))\pm (a(x)-b(x))\cos (\pi x_1) = 
%     2\pm\e(1+\cos(2\pi x_1)),\\
%        c(x)&=  (a(x)-b(x))\sin (\pi x_1) = \e\sin (2\pi x_1).
%\end{aligned}
% $$
Therefore $\bar F$ is 1-periodic in both 
$x_1$ and $x_2$, and  it projects to an 
analytic function $F$ on $\T^2$.
%The cocycles $\bar F$ and $\bar A$  over $\bar f$ are conjugate by analytic $\bar C$. 
For small $\epsilon$, $\bar F$ and 
$F$ are close to orthogonal, hence  are fiber bunched. 

Since $a(0)\ne b(0)$, the functions $a(x)$ and $b(x)$ are 
not cohomologous, and hence the coordinate line bundles 
$\E_1$ and $\E_2$ are the only invariant sub-bundles for $\bar A$ 
measurable with respect to $\bar \mu$ \cite[Lemma 7.1]{S11}. 
It is easy to see that  $\bar A$ has  one Lyapunov exponent
$$
\lambda=\lim_{n\to \infty}\frac{\log(a(x)\dots a(f^{n-1}x))}{n}=
\int_{\bar \T^2} \log a(x)\, d\bar \mu = \int_{\bar \T^2} \log b(x)\, d\bar \mu
\quad \text{for }\bar \mu \text{-a.e. }  x.
$$
Since $\bar F$ is conjugate by $\bar C$ to $\bar A$, it also has one 
Lyapunov exponent with respect to $\bar \mu$, and hence so does $F$.
However, $F$ has two exponents at $0=f(0)$, $\, \log (1+\e)$ 
and  $\log (1-\e)$, so it cannot preserve 
a conformal structure. Also, $\bar F$ preserves exactly two sub-bundles 
$\bar U^i =\bar C \, \E_i$. Their projections to $\T^2$ are not sub-bundles, 
as they ``twist together into a single object". 
(On the intermediate cover $\tilde \T^2=\R^2 / (2\Z \times \Z)$ this object 
splits into two non-orientable invariant sub-bundles, illustrating the lift
in the proof of Theorem \ref{reduction}.) We conclude that $F$ has no invariant 
sub-bundle, since lifting one to $\bar T^2$ would give an invariant  sub-bundle
for $\bar F$ different from $\bar U^1$ and $\bar U^2$. 

%The projective action of $\bar C$ is well defined on $\tilde \T^2$ and hence $\bar U^1$ and $\bar U^2$ project to the (non-orientable) sub-bundles $\tilde U^1$ and $\tilde U^2$ invariant under cocycle $\tilde F$ over $\tilde f$ on $\tilde \T^2$.  

\subsection{Subadditive sequences of functions}

Proposition \ref{<0} plays a key role in the proof of 
Corollary \ref{reductioncor}. It removes extra assumptions from
 \cite[Proposition 3.5]{RH}, which was similar to a result in 
 \cite{Sch}, but proved to be more useful for many applications. 

Let $f$ be a homeomorphism of a compact metric space $X$.
A sequence of continuous functions $a_n : X \to \R$ is called
{\em subadditive}\, if 
\begin{equation} \label{subad}
  a_{n+k}(x) \le a_k(x) +  a_n(f^k x) 
\quad\text{for all }x\in X \text{ and  }n,k\in \N.
\end{equation}
For any Borel probability measure $\mu$ on $X$ we denote
$a_n(\mu)= \int_X a_n d\mu$. If $\mu$ if $f$-invariant, \eqref{subad}
implies that $a_{n+k}(\mu) \le a_n(\mu) + a_k(\mu)$, i.e. the sequence 
of real numbers $\{ a_n (\mu) \}$ is subadditive. It is well known that
for such a sequence the following limit exists:
$$
 \chi (\mu) :=\lim_{n\to \infty} \frac {a_n (\mu)}n 
 = \inf_{n\in \N} \frac {a_n (\mu)}n.
$$
Also, by the Subaddititive Ergodic Theorem, if $\mu$ is ergodic 
then 
\begin{equation} \label{chi}
 \lim_{n\to \infty} \frac {a_n (x)}n= \chi (\mu) \quad
  \text{for $\mu$-almost all } x\in X.
\end{equation}

\begin{proposition}\label{<0} 
Let $f$ be a homeomorphism of a compact metric space $X$
and  $a_n : X \to \R$  be subadditive sequence of continuous functions.
If $\chi (\mu)<0$  for every ergodic invariant Borel probability measure 
$\mu$ for $f$, then there exists $N$ such that $a_N(x) <0$ for all $x \in X$.
\end{proposition}

\begin{proof}
We denote by $\m$ the set of  $f$-invariant Borel probability measures on $X$.
First we note that, by the Ergodic Decomposition, if $\chi (\mu)<0$  
for every {\em ergodic} $\mu\in \m$, then 
the same holds for every $\mu\in \m$.

First we show that there exists $K$ such that $a_K(\mu)<c<0$ for 
all $\mu\in \m$. 
Since $\chi(\mu)<0$ there exists $n_\mu$ and $c_\mu$ such that 
$a_{n_\mu}(\mu)<2c_\mu<0$. Since $a_n$ are continuous, for every
$\mu\in \m$ there is a neighborhood $V_\mu$  in the weak$^*$
topology such that $a_{n_\mu}(\nu)<c_\mu$ for every $\nu\in V_\mu$.
We choose a finite cover $\{V_{\mu_i}, \;i\in I\}$ of $\m$ 
and set $R=\max_I n_i$ and $c=\max_I c_i$. 
Take any $\mu\in \m$ and let $i$ be such that $\mu\in V_{\mu_i}$.
For any $K$ we can write $K=k n_i+r$, where $r < n_i \le R$, 
and by the subadditivity we get 
$$
  a_K(\mu)\le ka_{n_i}(\mu)+a_r (\mu) < kc+a_r (\mu).
$$
Since $a_r$ are uniformly bounded for $r<R$, we conclude that 
$a_K(\mu)<c<0$ provided that $K$, and hence $k$, are sufficiently large.

\begin{lemma}\label{psi} 
Suppose that for a continuous function $\phi:X\to \R$, $\phi(\mu) < c$ for 
all $\mu\in \m$.
Then there exists $n_0$ such that $\frac 1n \sum_{i=0}^{n-1} \phi(f^ix)<c$ for all $x \in X$ and $n \ge n_0$.
\end{lemma}

\begin{proof}
Suppose on the contrary that there exit sequences  $x_j \in X$ and 
$n_j \to \infty$ such that $S_j =\frac 1{n_j} \sum_{i=0}^{n_j-1} \phi(f^ix) \ge c$.
Note that $S_j=\psi (\mu_j)$, where 
$\mu_j = \frac 1{n_j} \sum_{i=0}^{n_j-1} \delta (f^i x_j)$ is a probability 
measure. Using compactness of the set of probability measures on $X$
we may assume, by passing to a subsequence if necessary, that 
$\mu_j$ weak$^*$ converges to a probability measure $\mu$. 
Since the total variation norm $\|f_* \mu_j -\mu_j \| \le \frac 2{n_j}$
it follows that the limit $\mu$ is $f$-invariant. On the other hand
$\psi (\mu) = \lim \psi (\mu_j) \ge c$, which contradicts the assumption.
\end{proof}

Applying the lemma to a function $a_K$ satisfying 
$a_K(\mu)<c<0$ we conclude that there is 
$n_0$ such that
$\sum_{i=0}^{n-1} a_K(f^ix)<cn\,$  for all $n\ge n_0$ and $x\in X$. 
Let $n=Km\ge n_0.$
Using subadditivity repeatedly, for $i=0, \dots , K-1$  we obtain
$$
\begin{aligned}
a_n(f^ix) &\le a_K(f^i x)+a_{n-K}(f^{K+i}x) \le \dots \\
& \le a_K(f^ix)+a_{K}(f^{K+i}x)+\dots +a_{K}\left( f^{(m-1)K+i}x \right).
\end{aligned}
$$
Adding these $K$ inequalities, we get
$$
a_n(x)+a_n(fx)+\dots +a_n(f^{K-1}x) \,\le\,
\sum_{i=0}^{n-1} a_K(f^ix) \le cn.
$$
Let $N=n+K$. For  $i=0, \dots , K-1$, we obtain
$$
a_N(x)\le a_i(x)+a_{n+K-i}(f^ix) \le  a_i(x)+a_n(f^i x)+a_{K-i}(f^{n+i} x)=:
a_n(f^i x)+\Delta_i(x),
$$
where we set  $a_0(x)=0$.
Let $M=\max\{\, |\Delta_i (x) |:\; 0\le i \le K-1, \; x\in \M\} $.
Adding the inequalities,
we get
$$
K\cdot a_N(x) \le a_n(x)+a_n(fx)+\dots +a_n(f^{K-1}x) + KM \le cn +KM
$$
and hence $a_N(x)\le cn/K+M$. Since $c<0$, taking $m=n/K$ 
sufficiently large, we can ensure  that $a_N(x)<0$ for all $x$.
\end{proof}

%%%%%%%%%%%%%%%%%%%%%%%%%%%%

\subsection{Proof of Corollary \ref{reductioncor}}
The fiber bunching assumption is used in the proofs of 
Theorems \ref{structure}, \ref{distribution}, and \ref{reduction}
only to obtain the stable and unstable holonomies for $F$.
Thus it suffices to show that the assumption in the corollary also
ensures their existence. 
Applying Proposition \ref{K(x,n)} below with $\xi=0$ we obtain that
for every $\e>0$ there exists $C_\e$ such that
\begin{equation} \label{K(x,n) 0}
   \|F^n_x\| \cdot \| (F^n_x)^{-1} \|  \le C_{\e} e^{\e |n|} 
\quad\text{for all }x\in \M \text{ and  }n\in \Z.
\end{equation}
We consider $\e$ such that $e^\e <\min \nu(x)^{-\beta}$. 
Then \eqref{K(x,n) 0} implies that for all sufficiently large $n$ the 
cocycle $F^n$ over $f^n$ is fiber bunched and, by Proposition \ref{close to Id}, 
has the stable and unstable holonomies. The holonomies for both
$F^N$ and $F^{N+1}$ are also holonomies for $F^{N(N+1)}$ and hence 
coincide by uniqueness, Proposition \ref{close to Id} (c). This easily
implies that they are also holonomies for $F$. We only need to check
Definition~\ref{holonomies}(iii) for $n=1$ which follows from 
that for $n=N+1$ at $x$ and for $n=N$ at $fx$:
$$
\begin{aligned}
H^s_{xy} &= (F^{N+1}_y)^{-1}\circ H^s_{f^{N+1}x \,f^{N+1}y} \circ F^{N+1}_x =\\
&= (F_y)^{-1} \circ (F^{N}_{fy})^{-1} \circ H^s_{f^{N+1}x \,f^{N+1}y} \circ F^{N}_{fx} \circ F_x = (F_y)^{-1} \circ \, H^s_{fx \,fy} \circ F_x.
\end{aligned}
$$

\begin{proposition}\label{K(x,n)} 
Suppose that there exists $\xi \geq 0$ such that
$\lambda_+(F,\mu)-\lambda_-(F,\mu)\le \xi$
for every  ergodic $f$-invariant measure $\mu$. 
Then for any $\e>0$ there exists $C_{\e}$ such that 
\begin{equation} \label{K(x,n) gamma}
   K_F(x,n) = \,\|F^n_x\| \cdot \| (F^n_x)^{-1} \| 
   \le C_{\e} e^{(\xi+\e) |n|} 
\quad\text{for all }x\in \M \text{ and  }n\in \Z.
\end{equation}
\end{proposition}

\begin{proof}
The proof is similar to that of \cite[Proposition 2.1]{KS10}
but simpler due to the use of Proposition \ref{<0}.
For a given $\e >0$ we apply Proposition \ref{<0} to 
the functions
 $$
 a_n(x)=\log K_F(x,n) - (\xi+\e) n, \quad n\in \N.
 $$
It is easy to see from the definition of the quasiconformal 
distortion that  
\begin{equation} \label{K(x,n) mult}
K(x, n+k) \le K(x,k) \cdot K(f^kx, n) 
\quad\text{for every } x \in \M \text{ and } n, k \geq 0.
\end{equation}
 It follows that  the sequence of 
functions  $\{a_n\}$ is subadditive. 
\vskip.1cm

Let $\mu$ be an ergodic $f$-invariant measure.
As in \eqref{K,lambda} we obtain   that 
$$
\lim _{n\to \infty} n^{-1} \log {K(x,n)} =
\lambda_+(F,\mu )-\lambda_-(F,\mu )\le \xi \quad \text{for $\mu$-a.e. }x.
$$
It follows that  
$
\lim _{n\to \infty} n^{-1} {a_n(x)} \le -\e<0$
for $\mu$-a.e. $x$, and \eqref{chi} implies that 
$\chi(\mu)<0$ for the sequence $\{a_n\}$.
Hence  there exists $N_\e$ such that 
$a_{N_\e}(x)<0,$ i.e. $K(x,N_\e) \le e^{(\xi+\e) N_\e}$
for all $x \in \M$. 
The proposition follows for $n>0$ from \eqref{K(x,n) mult} by 
choosing $C_\e$ sufficiently large,
and for $n<0$ since $K(x,n)= K(f^nx,-n)$.
\end{proof}

%%%%%%%%%%%%%%%%%%%%%%%%%%%%%%%%%

\subsection{Proof of Corollary \ref{locally accessible}}
Let $\tau$ be the continuous $F$-invariant conformal structure
obtained in Theorem \ref{structure}. It follows as in the proof of
Proposition \ref{tauHinvariant} that
$\tau(y)=H^s_{xy} (\tau (x))$ for all $x\in \M$ and $y\in W^s_{loc}(x)$.
By  Proposition \ref{close to Id}, 
$\|H^s_{xy}-\I_{xy}\|\le C\,\dist(x,y)^\beta$, and hence by 
Lemma \ref{A*}
$\;\dist (\I_{xy}(\tau (x)), \tau(y)) \le C'\,\dist (x,y)^\beta$.
The same holds for all $y\in W^u_{loc}(x)$.
Now local $\alpha$-H\"older accessibility implies  that $\tau$ 
is $\alpha \beta$-H\"older  on $\M$. 

In the proof of Theorem \ref{distribution} we established that
the sub-bundle is invariant under the stable and unstable
holonomies. Thus by the same reason it is $\beta$-H\"older 
along $W^u$ and $W^s$, and $\alpha \beta$-H\"older on $\M$.
Since the lift $\tilde f$ in Theorem \ref{reduction} is also locally 
$\alpha$-H\"older accessible, we obtain the same regularity
for the sub-bundles $\tilde \E^i$ and the factor bundles 
$\tilde \E^i/\tilde \E^{i-1}$. In particular, the holonomies for the induced
cocycle on $\tilde \E^i/\tilde \E^{i-1}$ are $\beta$-H\"older along 
$\tilde W^u$ and $\tilde W^s$. As above, the conformal 
structure on $\tilde \E^i/\tilde \E^{i-1}$ is invariant under these 
holonomies and hence it is $\alpha \beta$-H\"older on $\tilde \M$.
$\QED$

%%%%%%%%%%%%%%%%%%%%%%%%%%%%%%%%%

\subsection{Proof of Corollary \ref{anosov}}
The difference in the proof for the partially hyperbolic and Anosov 
case is in obtaining global continuity on $\M$ from that along the stable 
and unstable foliations, and this argument is more direct in the Anosov case.

Theorem \ref{structure} for Anosov case follows from Proposition \ref{close to Id} (a) and
\cite[Proposition 2.3]{KS10}. Alternatively, as in Theorem \ref{structure}, 
we obtain that the conformal structure $\tau$ is essentially $\beta$-H\"older 
along $W^s$ and $W^u$ and then conclude
that it is essentially $\beta$-H\"older on $\M$ using $1$-H\"older 
accessibility and the local product structure of the measure,
as in the end of the proof of \cite[Proposition 2.3]{KS10}.
In the proof of Theorem \ref{distribution} we only need to replace 
the reference \cite[Theorem C]{ASV} by \cite[Theorem D]{AV}
and then conclude that the invariant distribution is
$\beta$-H\"older by Corollary \ref{locally accessible}.
Using these results, Corollary \ref{qc implies conf} and 
Theorem \ref{reduction} can be obtained as in this paper.

The prove the last statement of the corollary we recall that
for H\"older continuous linear cocycles over hyperbolic systems,
the Lyapunov exponents of any ergodic invariant measure can 
be approximated by Lyapunov exponents of periodic measures
\cite[Theorem 1.4]{K11}. Hence we obtain 
$\lambda_+(F,\eta)=\lambda_-(F,\eta)$ for every ergodic $f$-invariant 
measure $\eta$ and the argument in the proof of 
Corollary \ref{reductioncor} applies.   

%%%%%%%%%%%%%%%%%%%%%%%%%%%%%%%%%

\subsection{Proof of Theorem \ref{reductionH}}

By Corollary \ref{anosov} the conclusion of Theorem \ref{reduction} holds, 
moreover, the invariant sub-bundles and conformal structures are 
$\beta$-H\"older.

We use notations from the proof of Theorem \ref{reduction}.
Recall that $\tilde F : \tilde  \E \to \tilde  \E$ is the  lift of $F$ to the 
cocycle over $\tilde f :\tilde \M \to \tilde \M $, 
and that the cocycle $\hat F=\tilde F^N$ preserves the flags 
$\tilde \W^1, ..., \tilde \W^l$. We denote these $\beta$-H\"older flags by
$$\tilde \W^j= \left\{ \, \{0\}=\tilde \E^{j,0}\subset \tilde \E^{j,1} \subset \dots
 \subset \tilde \E^{j,k}=\tilde \E \right\}, \qquad j=1, ..., l.$$
 Now we define sub-bundles 
$\hat \E^{i} = \sum_{j=1}^l \tilde \E^{j,i}$, where the sum may be 
 not direct, which form a new $\beta$-H\"older flag, where the inclusions 
may be not strict. By the construction of $\tilde \W^j$, this flag projects to 
a $\beta$-H\"older $F$-invariant flag 
$\{ 0 \} =\E^0 \subset \E^1\subseteq \E^2 \subseteq \dots \subseteq \E^k= \E$.
Eliminating unnecessary sub-bundles in case of
equalities we obtain the desired flag \eqref{flagH}.
We will show below that for each $i=1, ..., k$ the factor bundle 
$\hat \E^i /\hat \E^{i-1}$ has a continuous conformal structure 
invariant under the factor of $\hat F$. This implies uniform 
quasiconformality of the factor cocycle $F^{(i)}$ induced by $F$ 
on $\E^i / \E^{i-1}$ and, by Corollaries \ref{qc implies conf} and \ref{anosov}, existence of a $\beta$-H\"older conformal structure $\tau_i$ on $\E^i / \E^{i-1}$ invariant under $F^{(i)}$, $i=1, ..., k$. Then the proof is completed as follows.

We normalize the conformal structure $\tau_i$ to get a $\beta$-H\"older 
Riemannian metric 
$\|.\|_i'$ on $\E^i / \E^{i-1}$, $i=1, ..., k$, with respect to which $F^{(i)}$ 
is conformal. We denote by $a_i(x)$ the scaling coefficients with respect to these norms:
$$ 
\| F^{(i)}(v) \|_i' = a_i(x) \| v\|_i' \quad \text{ for all } v\in \E^{i}/\E^{i-1}\, (x).
$$
These are positive $\beta$-H\"older functions, and for each periodic point 
$p=f^np$ the product $a_i(f^{n-1}p) \cdots a_i(fp) \, a_i(p)$ is the same 
for all $i$. Indeed, the matrix of $F^n_p$ is block triangular in any basis of 
$\E_p$ appropriate for the flag, and the above product is simply the modulus of 
the eigenvalues of the block corresponding to $\E^i / \E^{i-1}$; all these
moduli are equal since $F^n_p$ has only one Lyapunov exponent. 
Now by the Liv\v{s}ic theorem \cite{Liv2},\cite[Theorem 19.2.1]{KH} this implies that the functions $a_i$ are cohomologous, more precisely
there exist positive $\beta$-H\"older functions $\psi_i$ such that for all $x$
$$
a_i(x)= a_1 (x) \psi_i (fx) \psi_i (x)^{-1}, \quad i=2, ..., k.
$$ 
Choosing new metric $\| .\|_i = \psi_i ^{-1} \| .\|_i'$, $i=2, ..., k$, 
makes $a_1$ the scaling coefficient for all $F^{(i)}$.  
Hence the cocycle $a_1(x)^{-1} F(x)$  induces 
isometries on each $\E^{i}/\E^{i-1}$, $i=1, ..., k$.

%%%%%%%%%%%
It remains to obtain conformal structures on $\hat \E^i / \hat \E^{i-1}$ 
invariant under the factors of $\hat F$. We fix $1 \le i \le k$. As we 
showed in the proof of Theorem \ref{reduction}, for each 
$j=1, ..., l$ the factor of $\hat F$ on $\tilde \E^{j,i}/\tilde \E^{j,i-1}$ has 
an invariant conformal structure, which in our case is 
$\beta$-H\"older. We normalize these structures to obtain $\beta$-H\"older 
Riemannian metrics $g_j$. As in the argument above we can show that 
the scaling coefficients of $\hat F$ are 
cohomologous functions. Hence we may
assume that the metrics $g_j$ are normalized so that for some positive
function $\varphi$ the cocycle $\bar F =\varphi \hat F$ induces an 
isometry on $\tilde \E^{j,i}/\tilde \E^{j,i-1}$ for each $j=1, ..., l$. 
To simplify notations we will also write $\bar F$ for its induced map
on any factor bundle.

We fix $j$ and consider 
$\bar \E^j =\tilde \E^{j,i}/(\tilde \E^{j,i} \cap \hat \E^{i-1})$ as a factor bundle 
of $\tilde \E^{j,i}/\tilde \E^{j,i-1}$. Since $\bar F$ is isometric on 
$\tilde \E^{j,i}/\tilde \E^{j,i-1}$, it preserves the orthogonal complement 
of $(\tilde \E^{j,i} \cap \hat \E^{i-1})/\tilde \E^{j,i-1}$ and the metric $g_i$ 
restricted to it. This orthogonal complement is isomorphic  
to $\bar \E^j$, and thus we obtain an $\bar F$-invariant metric $\bar g_j$ 
on $\bar \E^j$. Now we view $\bar \E^j$ as a sub-bundle 
of $\hat \E^{i}/\hat \E^{i-1}$, so that $\hat \E^{i}/\hat \E^{i-1} =\sum_{j=1}^l \bar \E^j$, and combine the metrics $\bar g_j$ as follows.
Let $\U=\bar \E^1 \cap \bar \E^2$ and $\U^\perp$ be its orthogonal
complement in $\bar \E^2$. As before, $\bar F$ preserves $\U^\perp$ 
and the restriction of $\bar g_2$ to it. 
We combine $g_1$ and $g_2$ into $\bar F$-invariant the Riemannian 
metric on $\bar \E^1 + \bar \E^2=\bar \E^1 \oplus \U^\perp$ by declaring 
the last two bundles orthogonal. Continuing this inductively we obtain
a $\beta$-H\"older Riemannian metric on $\hat \E^{i}/\hat \E^{i-1}$
with respect to which $\bar F$ is isometric and $\hat F$ conformal.

%%%%%%%%%%%%%%%%%%%%%%%%%%%%%%%%%

\subsection{Proof of Theorem \ref{polynomial}}

By Theorem \ref{reductionH} the cocycle $G(x)=\phi (x) F(x)$ induces 
isometries  on each factor bundle $\E^{i}/\E^{i-1}$.  
Inductive application of the next proposition shows that 
$\| G^n(x) \| \le Dn^{k-1}$. Applying it to $G^{-1}$  yields
$\| G^{-n}(x) \| \le Dn^{k-1}$, and hence the quasiconformal distortion 
satisfies $K_{G}(x,n)\le C n^{2(k-1)}$.
Since $K_F(x,n)=K_G(x,n)$, the first part of the theorem follows.

If $\lambda_+(F,\mu_p )=\lambda_-(F,\mu_p )=0$, then $\phi (x)$ is cohomologous to the constant 1 and, by rescaling the norm, we obtain that 
$F$ itself induces isometries on each factor bundle 
$\E^{i}/\E^{i-1}$. Hence the second part  also follows from the next proposition.

\begin{proposition}\label{ind} 
Let $F : \E \to \E$ be a continuous linear cocycle over $f$ and let
$V$ be an $F$-invariant continuous sub-bundle. Suppose that
the factor cocycle $\bar F : \E/V \to \E/V$ is an isometry 
and that for some $C$ and $k$ the  
restriction $F_V=F|_V$ satisfies $\|F_V^n(x)\| \le Cn^j$ for all 
$x$ and $n\in \N$. Then there exists a constant $D$ such that
$\| F^n(x) \| \le Dn^{j+1}$ for all $x$ and $n\in \N$.
\end{proposition}

\begin{proof}
We denote by $P:\E \to \E/V$ the natural projection and by $\pi :\E \to V$ 
the orthogonal projection with respect to some Riemannian metric on $\E$. 
Then for any $x$ the map $v \mapsto (P(v),\pi (v))$ identifies $\E(x)$ with 
$\E/V (x) \oplus V(x)$, and $\max \{ \|(P(v)\|,\| \pi (v)\| \}$ gives a convenient 
continuous norm on $\E$.  
Since the linear map 
$$
\Delta '(x) =(\pi \circ F - F_V \circ \pi)(x): \E(x) \to V(fx)
$$
is identically zero on $V(x)$ we can write it as 
$\Delta ' (x)= \Delta (x)  \circ P $, where the linear map 
$\Delta (x): \E/V (x) \to V(fx)$ depends continuously on $x$.
Thus we have
$$
\pi \circ F = F_V \circ \pi  + \Delta \circ P, \quad  
P \circ F^n = \bar F^n \circ P,  \qquad \text{and hence}
$$
$$
\begin{aligned}
\pi \circ F^n & =  (F_V \circ \pi  + \Delta \circ P)  \circ F^{n-1}  =
F_V \circ (\pi \circ F^{n-1})  + \Delta \circ \bar F^{n-1} \circ P = \\
& = F_V \circ ((F_V \circ \pi  + \Delta \circ P) \circ F^{n-2})  + \Delta \circ \bar F^{n-1} \circ P = \\
& = F_V^2 \circ ( \pi \circ F^{n-2})+ F_V\circ \Delta  \circ \bar F^{n-2}  \circ P 
+ \Delta \circ \bar F^{n-1} \circ P = \dots = \\
& = F^n_V  \circ \pi  + 
\sum_{i=0}^{n-1}  F^{n-i-1}_V  \circ \Delta  \circ \bar F^i \circ P.
\end{aligned}
$$
Let $K$ be such that $\| \Delta (x) \| \le K$ for all $x$. 
Since  by the assumptions $\|F_V^i\| \le Cn^i$ and $\|\bar F^i \|=1$,
we can estimate 
$$
\| \pi \circ F^n (x)\| \le Cn^j  + 
\sum_{i=0}^{n-1}  C(n-i-1)^j \cdot K \le Cn^j  + n Cn^j K \le D n^{j+1}
$$
for some constant $D$ independent of $n$ and $x$. 
Since $\| P \circ F^n  \| = \|\bar F^n \circ P \| \le 1$,  we conclude that 
$$ 
\| F^n (x) \| = \max \,\{\| P \circ F^n (x)\| ,\, \| \pi \circ F^n (x)\| \} \le D n^{j+1}.
$$
\end{proof}

%%%%%%%%%%%%%%%%%%%%%%%%%%%%%%

%%%%%%%%%%%%%%%%%%%%%%%%%%%%%%%%%
%%%%%%%%               bibliography                %%%%%%%%%
%\newpage


\begin{thebibliography}{XXX}

\bibitem[AV]{AV} A. Avila, M. Viana. 
           {\em Extremal Lyapunov exponents: an invariance principle 
           and applications.} To appear in Inventiones Mathematicae.

\bibitem[ASV]{ASV} A. Avila, J. Santamaria, M. Viana. 
              {\em Cocycles over partially hyperbolic maps.} Preprint.

\bibitem[BP]{BP} L. Barreira, Ya. Pesin.  
              {\em Nonuniform Hyperbolicity: 
              Dynamics of systems with nonzero Lyapunov exponents.}
              Encyclopedia of Mathematics and Its Applications, 
              {\bf 115} Cambridge University Press.

\bibitem[BW]{BW}  K. Burns, A. Wilkinson. {\em On the ergodicity 
of partially hyperbolic systems.}
Annals of Mathematics, 171 (2010) 451-489.

\bibitem[GKS]{GKS11} A. Gogolev, B. Kalinin, V. Sadovskaya. {\em Local rigidity for Anosov automorphisms.} Mathematical Research Letters, {\bf 18} (2011), no. 5, to appear.



\bibitem[HKt]{HKt} S. Hurder, A. Katok. 
             {\em Ergodic theory and Weil measures for foliations.} 
             Annals of Mathematics, (2) 126 (1987), no. 2, 221-275.

\bibitem[K]{K11} B. Kalinin. 
            {\em Liv\v{s}ic theorem for matrix cocycles.} 
            Annals of Mathematics, 173 (2011), no. 2, 1025-1042. 


\bibitem[KS09]{KS9} B. Kalinin, V. Sadovskaya.  
            {\em On Anosov diffeomorphisms with asymptotically conformal 
            periodic data}.  Ergodic Theory Dynam. Systems, 29 (2009), 117-136. 

\bibitem[KS10]{KS10} B. Kalinin, V. Sadovskaya.  
           {\em Linear cocycles over hyperbolic systems and criteria 
            of conformality.}
           Journal of Modern Dynamics, vol. 4 (2010), no. 3, 419-441.
           
          
\bibitem[KtH]{KH}  A. Katok, B. Hasselblatt. 
              {\em Introduction to the modern theory of dynamical systems.}  
              Encyclopedia of Math. and  its Applications,
              vol. 54. Cambridfe University Press, London-New York, 1995.




\bibitem[L]{Liv2} A. N. Liv\v{s}ic. {\em Cohomology of dynamical systems.} 
                      Math. USSR Izvestija 6, 1278-1301, 1972. 

\bibitem[Ma]{M}  H. Maass.
{\em Siegel's modular forms and Dirichlet series.} 
Lecture Notes in Mathematics, Vol. 216. Springer-Verlag, Berlin-New York, 1971.               

\bibitem[M]{Moore} C. Moore. {\em Amenable subgroups of semi-simple groups
and proximal flows.}   Israel Journal of Mathematics, Vol 34, 121-138,  1979.

\bibitem[NT]{NT} V. Nitica, A. T\"or\"ok. {\em Regularity of the transfer map for 
                          cohomologous cocycles.} 
                          Ergodic Theory Dynam. Systems, 18(5), 1187-1209, 1998. 
         
\bibitem[PW]{PW} M. Pollicott, C. P. Walkden.  {\em Liv\v{s}ic theorems for connected Lie 
groups.} Trans. Amer. Math. Soc., 353(7), 2879-2895, 2001.                       
                                                 
 \bibitem[R]{RH} F. Rodriguez Hertz. 
              {\em Global rigidity of certain abelian actions 
             by toral automorphisms.}  Journal of Modern Dynamics, 
             Vol. 1, no. 3 (2007), 425-442.
             

\bibitem[S]{S11}   V. Sadovskaya. {\em Cohomology of 
               $GL(2,\R)$-valued cocycles over hyperbolic systems.} Preprint.

\bibitem[Sch]{Sch} S. J. Schreiber. {\em On growth rates of subadditive functions for semi-flows,}
              J. Differential Equations, 148, 334Ð350, 1998.


\bibitem[Su]{Su}   D. Sullivan. {\em On the ergodic theory at infinity of an 
              arbitrary discrete group of hyperbolic motions.}
              In ``Riemann surfaces and related topics'',
              Annals of Math. Studies, {\bf 97} (1981), 465-497.

\bibitem[T]{T}   P. Tukia. {\em On quasiconformal groups.}
              Jour. d'Analyse Mathe. {\bf 46} (1986), 318-346. 
              
\bibitem[V]{V} M. Viana. 
           {\em Almost all cocycles over any hyperbolic system have 
           nonvanishing Lyapunov exponents.} 
           Ann. of Math. (2) 167 (2008), no. 2, 643--680. 
                         
\bibitem[W]{W}  A. Wilkinson. {\em The cohomological equation for partially hyperbolic diffeomorphisms.} Preprint.  


\end{thebibliography}
\end{document}